\documentclass{svjour3}                     
\smartqed  
\usepackage{bbm}
\usepackage{mathrsfs}
\usepackage{amsfonts}
\usepackage{stmaryrd}
\usepackage{graphicx}
\usepackage{subfigure}

\usepackage{bm}
\usepackage{amsmath}
\usepackage{amssymb}

\usepackage{amstext}
\usepackage{amsbsy}
\usepackage{setspace}
\usepackage{algorithm}
\usepackage{algorithmic}

\numberwithin{equation}{section}
\numberwithin{theorem}{section}
\numberwithin{lemma}{section}
\numberwithin{remark}{section}
\usepackage{mathptmx}

\newtheorem{thm}{Theorem}
\newtheorem{lem}[thm]{Lemma}

\newtheorem{defn}{Definition}

\newtheorem{rem}{Remark}

\usepackage{mathptmx}      

\begin{document}

\title{A hybridized discontinuous Galerkin method for 2D fractional convection-diffusion equations}


\author{Shuqin Wang, Jinyun Yuan, \\ Weihua Deng, Yujiang Wu
}


\institute{
Shuqin Wang \at
              School of Mathematics and Statistics, Gansu Key Laboratory of Applied Mathematics and Complex
Systems, Lanzhou University, Lanzhou 730000, People's Republic of China.  \\
            Department of Mathematics, Federal University of Paran\'a, Centro Polit\'{e}cnico, CP: 19.089, Curitiba, CEP: 81531-980, PR, Brazil.
            \email{wsqlzu@gmail.com} 
           \and
         Jinyun Yuan \at
         Department of Mathematics, Federal University of Paran\'a, Centro Polit\'{e}cnico, CP: 19.089, Curitiba, CEP: 81531-980, PR, Brazil.
         \email{yuanjy@gmail.com}
         \and
         Weihua Deng\at School of Mathematics and Statistics, Gansu Key Laboratory of Applied Mathematics and Complex
Systems, Lanzhou University, Lanzhou 730000, People's Republic of China.
\email{dengwh@lzu.edu.cn}
\and
         Yujiang Wu\at School of Mathematics and Statistics, Gansu Key Laboratory of Applied Mathematics and Complex
Systems, Lanzhou University, Lanzhou 730000, People's Republic of China.
\email{myjaw@lzu.edu.cn}
}

\date{Received: date / Accepted: date}
\maketitle

\begin{abstract}
A hybridized discontinuous Galerkin method is proposed for solving 2D fractional  convection-diffusion equations containing derivatives of fractional order in space on a finite domain.
The Riemann-Liouville derivative is used for the spatial derivative. Combining the characteristic method and the hybridized discontinuous Galerkin method, the symmetric variational formulation is constructed. The stability of the presented scheme is proved. Theoretically, the order of $\mathcal{O}(h^{k+1/2}+\Delta t)$ is established for the corresponding models and numerically the better convergence rates are detected by carefully choosing the numerical fluxes.
Extensive numerical experiments are performed to illustrate the performance of the proposed schemes. The first
numerical example is to display the convergence
orders, while the second one justifies the benefits of the schemes. Both are tested with triangular meshes.
\\
\\
\noindent{\bf AMS} 26A33, 35R11, 65M60, 65M12.
\end{abstract}
\section{Introduction}
Fractional differential equations (FDEs) have become more and more popular in applied science and engineering
field recently. The history and mathematical background of fractional differential operators are given in \cite{r24} with definitions and applications of fractional calculus. This kind of equations has been used increasingly in many fields, for example, in Nature \cite{r25} fractional operators applied in fractal stream chemistry and its  implications for contaminant transport in catchments,
in \cite{r26} the fractional calculus motivated into bioengineering, and its application as a model for physical phenomena exhibiting anomalous diffusion, L$\acute{e}$vy motion, turbulence
\cite{r27,r28,r29}, etc.

 Let us briefly review the development of numerical methods for the fractional convection-diffusion equations. Several authors have proposed a variety of high-order finite difference schemes for solving time-fractional
 convection-diffusion equations, for example \cite{r16,r23,r18,r15}, and solving space-fractional convection-diffusion equations \cite{r11,r14}. In \cite{r20,r21,r19}, W. Mclean and K. Mustapha have used the piecewise-constant and piecewise-linear discontinuous Galerkin (DG) methods to solve the time-fractional diffusion and wave
 equations, respectively. But these methods require more computational costs. In order to tackle those problems, in \cite{r22} W. Mclean has proposed an efficient scheme called fast summation by interval clustering to reduce the implementation memory; more recent works on this issue can been in \cite{r6-1,r19-1}. Furthermore, in \cite{r9} Deng and Hesthaven have developed DG methods for fractional
 spatial derivatives and given a fundamental frame to combine the DG  methods with fractional operators. In \cite{r10} Xu and Hesthaven have applied the DG methods to the fractional convection-diffusion equations in one dimension. In the two dimensional case, Ji and Tang \cite{r7} have applied the DG methods to recast the fractional diffusion equations in rectangular meshes with the numerically optimal convergence order $\mathcal{O}(h^{k+1})$. However, there are no theoretical results. So far very few literatures deal with the fractional problems in triangular meshes, besides \cite{r30}. This motives us to consider a successful DG method for solving the fractional problems in triangular meshes.

Here, we consider the time-dependent space-fractional convection-diffusion problem
\begin{equation}\label{eq1}
\centering
  \begin{cases}
\partial_t u+\bm b\cdot\nabla u-\frac{\partial^\alpha u}{\partial x^\alpha}-\frac{\partial^\beta u}
  {\partial y^\beta}=f, & (x,y,t)\in\Omega\times J,\cr
  u(x,y,0)=u_{0}(x,y),& (x,y) \in\Omega,\cr
   u(x,y,t)=0,& (x,y,t)\in\partial
   \Omega\times J,
\end{cases}
\end{equation}
in the domain $\Omega=(a,b)\times(c,d)$ and $J=[0,T]$ with the superdiffusion operators which are defined by the left Riemann-Liouville fractional derivatives
$\frac{\partial^\alpha u}{\partial x^\alpha}$ and $\frac{\partial^\beta u}
{\partial y^\beta}$, $1<\alpha,\beta<2$.
The function
$f\in L^2(J;L^2(\Omega))$ is a source term; the convection coefficient
 $\bm b$ is supposed to satisfy $\bm b\in L^\infty(J;W^{1,\infty}(\Omega)^2)$, and the initial function $u\in L^2(\Omega)$.

In this work, we shall design a stable and accurate DG method for (\ref{eq1}). The stability and convergence  are proved in multi-dimensional case. This development is built on the extension of the previous DG works found in \cite{r9,r10}, where a qualitative study of the high-order local DG methods was discussed and some theoretical results were offered
in one space dimension. In order to perform the error analysis, the authors defined
some projection operators to prove the convergence results. Unfortunately, the defined projection operators can not be easily extended to two dimensional case (see \cite{r9,r10}). Hence, to avoid this difficulty, a different DG method is designed in this paper by carefully choosing the numerical fluxes and adding penalty terms. The presented hybridized discontinuous Galerkin (HDG) method has the following attractive properties:
1) The HDG method can be used for other fractional problems, for example, fractional diffusion equations; 2) It has excellent provable stability, i.e., the stability can be proved in any space dimensions; 3) Theoretically, the error analysis can be more easily performed with the general analytical methods in any space dimensions.

The outline of this paper is as follows. In Section 2, we introduce some basic definitions, notations and review a few lemmas which are useful for the following analysis. In Section 3, we present the computational schemes and give some discussions. In Section 4, we perform the stability and convergence analysis for the 2D space-fractional convection-diffusion equations. In Section 5, we make the numerical experiments and show some simulation results to verify the theoretical results and illustrate the performance of the proposed schemes. We conclude the paper with some remarks in the last section.


\section{Preliminaries}

In the following we give some definitions of fractional integrals, derivatives, and their properties.
\begin{defn}[\cite{r24}]\label{definition2.1}
 For any $\mu>0$, the left and right Riemann-Liouville fractional
 integrals of function $u(x)$ defined on $(a,b)$ are defined by
\begin{align*}
_a\!I_x^{\mu}u(x)=\int_a^x\frac{(x-\xi)^{\mu-1}}
{\Gamma(\mu)}u(\xi)d\xi,
\end{align*}
and
\begin{align*}
_x\!I_b^{\mu}u(x)=\int_x^b\frac{(\xi-x)^{\mu-1}}
{\Gamma(\mu)}u(\xi)d\xi.
\end{align*}
\end{defn}

\begin{defn}[\cite{r24}]\label{definition2.2}
 For any $\mu>0, n-1<\mu< n, n\in N^+$,
the left and right Riemann-Liouville fractional derivatives of function $u$ defined on $(a,b)$ are defined by
\begin{align*}
_{a}\!D_x^\mu u(x)=\frac{d^n}{dx^n}
\int_a^x\frac{(x-\xi)^{n-\mu-1}}{\Gamma(n-\mu)}u(\xi)d\xi,
\end{align*}
and
\begin{align*}
~~~~~~~~~~~{_{x}\!}D_{b}^{\mu}u(x)=(-1)^n\frac{d^n}{dx^n}
\int_x^b\frac{(\xi-x)^{n-\mu-1}}{\Gamma(n-\mu)}u(\xi)d\xi.
\end{align*}
\end{defn}

\begin{defn}[\cite{r24}]\label{definition2.5}
 For any $\mu>0, n-1<\mu< n, n\in N^+$, Caputo's left and right fractional
derivatives of function $u(x)$ on $(a,b)$ are defined by
\begin{align*}
_{a}^{C}\!D_{x}^{\mu}u(x)&=
\int_{a}^{x}\frac{(x-\xi)^{{n-\mu-1}}}{\Gamma(n-\mu)}
\frac{d^n u(\xi)}{d\xi^n}d\xi,
\end{align*}
and
\begin{align*}
~~~~~~~~_{x}^{C}\!D_{b}^{\mu}u(x)&=
\int_{x}^{b}\frac{(\xi-x)^{{n-\mu-1}}}{\Gamma(n-\mu)}
\frac{(-1)^nd^n u(\xi)}{d\xi^n}d\xi.
\end{align*}
\end{defn}

\begin{lem}[Adjoint property \cite{r9,r13,r10}]\label{lemma2.3}
 For any $\mu>0$, the left and right
Riemann-Liouville fractional integral operators are adjoints for any functions $u(x),v(x)\in L^2(a,b)$, i.e.,
\begin{align}
\int_a^b {_a}\!I_x^{\mu}u(x)v(x)dx&=\int_a^b u(x){_x}\!I
_b^{\mu}v(x)dx.
\end{align}
\end{lem}


\begin{lem}[\cite{r9,r10}]\label{lemma2.6}
 Suppose that $u(x)$ is a function
defined on $(a,b)$, $u^{(k)}(x)=0$ when $x=a$ or $x=b$, $\forall\ 0\leq k\leq n-1~(n-1<\mu< n),\, n\in N^+$. There are
$$ _a\!D^\mu_x u(x)=D^n{_a\!I^{n-\mu}_x}
u(x)={_a\!I^{n-\mu}_x}\big(D^nu(x)\big),$$
or
$$_x\!D^\mu_b u(x)=(-D)^n{ _x\!I^{n-\mu}_b}
u(x)=_x\!I^{n-\mu}_b\big((-D)^n u(x)\big).$$
\end{lem}

Note that, from Definition \ref{definition2.2}, Definition \ref{definition2.5} and Lemma \ref{lemma2.6} if the solution $u$ of (\ref{eq1}) satisfies $u(x,y)=0$ when $x=a$ or $y=c$, then for any $1<\alpha,\beta<2$, the left fractional Riemann-Liouville derivatives of function $u(x,y)$
on $\Omega=(a,b)\times(c,d)$ can be rewritten as (see \cite{r9,r10}):
\begin{align}\label{equation2.4}
\frac{\partial^\alpha u}{\partial x^\alpha}=\frac{\partial}{\partial x}{_a\! I^{2-\alpha}_x } \big(\frac{\partial}{\partial x}u(x,y)\big),
\end{align}
\begin{align}\label{equation2.5}
\frac{\partial^\beta u}{\partial y^\beta}=\frac{-\partial}{\partial y}{_c\! I^{{2-\beta}}_y } \big(\frac{-\partial}{\partial y}u(x,y)\big).
\end{align}
For the convenience, we use the notation
\begin{equation}\label{eq2.6}
I^{\bar{\bm\alpha}}_{\bm x}=({ _a\!I^{\alpha_1}_x},{ _c\!I^{\alpha_2}_y}),
\end{equation}
 where $(\alpha_1,\alpha_2)=(2-\alpha,2-\beta)$ and $\alpha_1,\alpha_2\in(0,1)$.

\begin{defn}[The left and right fractional spaces
\cite{r9}] \label{definition2.7}
For $0<\mu<1$,
extend $u(x)$ outside of
$\mathfrak{I}:=(a,b)$ by zero. Define the norms
\begin{align}
\parallel u\parallel_{J_L^{-\mu}(\mathbb{R})}&:= \parallel _{-\infty}\!I_x^\mu u\parallel_{L^2(\mathbb{R})},
\end{align}
\begin{align}
\parallel u\parallel_{J_R^{-\mu}(\mathbb{R})}&:= \parallel _{x}\!I_\infty^\mu u\parallel_{L^2(\mathbb{R})}.
\end{align}
\end{defn}

Let the two spaces $J_{L}^{-\mu}(\mathbb{R})$ and $J_{R}^{-\mu}(\mathbb{R})$ denote the
closures of $C_{0}^\infty(\mathbb{R})$ with respect to $\parallel\cdot
\parallel_{J_L^{-\mu}}$ and $\parallel\cdot\parallel_{J_R^{-\mu}}$,
respectively.

\begin{lem}[\cite{r9,r13,r10}]\label{lemma2.8}
 For $\mu>0$, assume that $u(x)$ is a real function. Then
\begin{align}
(_{-\infty}\!I^\mu_x u,_{x}\!I^\mu_\infty u)
=cos(\mu\pi)\parallel u\parallel^2_{J_L^{-\mu}(\mathbb{R})}=cos(\mu\pi)\parallel u\parallel^2_{J_R^{-\mu}(\mathbb{R})}.
\end{align}
\end{lem}

Generally, we consider the case in which the problem is in a bounded domain instead of $\mathbb{R}$. Thus we
restrict the definitions to $\mathfrak{I}=(a,b)$.
\begin{defn}[\cite{r9,r10}]\label{definition2.9}
 Define the spaces $J_{L,0}^{-\mu}(\mathbb{\mathfrak{I}})$ and $J_{R,0}^{-\mu}(\mathfrak{I})$ as the closures
of  $C_{0}^{\infty}(\mathfrak{I})$ under their respective norms.
\end{defn}

\begin{thm}[\cite{r9,r10}]\label{theorem2.10}
  If  $-\mu_2<-\mu_1<0$, then $J_{L,0}^{-\mu_1}(\mathfrak{I})$ and $J_{R,0}^{-\mu_1}(\mathfrak{I})$ are embedded into $J_{L,0}^{-\mu_2}(\mathfrak{I})$ and $J_{R,0}^{-\mu_2}(\mathfrak{I})$, respectively. Furthermore, $L^2(\mathfrak{I})$ is embedded into both of them.
\end{thm}

\begin{defn}[\cite{r9,r10}]\label{definition2.11}
By Lemma \ref{lemma2.3}, Lemma \ref{lemma2.8}, Definition \ref{definition2.7} and Definition \ref{definition2.9}, we obtain
\begin{align}
\int_c^d(_{a}\!I_{x}^{\alpha_1}u(\cdot,y),u(\cdot,y))_{L^2(a,b)}dy&=cos(\alpha_1\pi/2)
\int_c^d \parallel u(\cdot,y)\parallel^{_2}_{J_{R,0}^{-\alpha_1/2}(a,b)}dy,
\\
\int_a^b(_{c}\!I_{y}^{\alpha_2}u(x,\cdot),u(x,\cdot))_{L^2(c,d)}dx&=cos(\alpha_2\pi/2)\int_a^b \parallel u(x,\cdot)\parallel^{_2}_{J_{R,0}^{-\alpha_2/2}(c,d)}dx.
\end{align}
Let the spaces $J_{R,0}^{-\alpha_1/2}(a,b)$ and $J_{R,0}^{-\alpha_2/2}(c,d)$ denote
the closures of $C^{\infty}_{0}(a,b)$ and $C^{\infty}_{0}(c,d)$ under their respective norms, and $\alpha_1=2-\alpha,\alpha_2=2-\beta$, $\alpha_1,\alpha_2\in(0,1)$.
\end{defn}

\section{Derivation of the numerical schemes}
We first review some notations, and then focus on deriving the fully discrete numerical scheme of the 2D
space-fractional convection-diffusion equation.

\subsection{Notations}
For the mathematical setting of the DG methods, we
describe some spaces and notations. The domain $\Omega$ is subdivided into elements $E$. Here $E$ is a triangle in 2D. We assume that the intersection of two elements is either empty, or an edge (2D). The mesh is called regular if

$$\forall E\in\mathscr{E}_h,~~~\frac{h_E}{\rho_E}\leq C,$$
where $\mathscr{E}_h$ is the subdivision of $\Omega$, $C$ a constant, $h_E$ the diameter of the element $E$, and $\rho_E$ the diameter of the inscribed circle in element $E$. Throughout this work $h=\max_{E\in\mathscr{E}_h}h_E$.

We introduce the broken Sobolev space for any real numbers $s$ by
$$H^s(\mathscr{E}_h)=\left\{ v\in L^2(\Omega):\,\forall E\in\mathscr{E}_h,v|_E\in H^s(E)\right\},$$
equipped with the broken Sobolev norm:
$$\parallel v\parallel_{H^s(\mathscr{E}_h)}=\big(\sum_{E\in\mathscr{E}_h}\parallel v\parallel^2_{H^s(E)}\big)^{\frac{1}{2}}.$$

The set of edges of the subdivision $\mathscr{E}_h$ is denoted by $\mathscr{E}_h^B$. Let  $\mathscr{E}_h^i$ denote the set of interior edges, and $\mathscr{E}_h^b=\mathscr{E}_h^B\ \backslash  \mathscr{E}_h^i$ denote the set of edges on $\partial\Omega$. With each edge $e$, the unit normal vector is $\bm n_e$. If $e$ is on the boundary $\partial\Omega$, then $\bm n_e$ is taken to be the unit outward vector normal to $\partial\Omega$ \cite{r17}.

If $v$ belongs to $H^1(\mathscr{E}_h)$, then the trace of $v$ along any side of one element $E$ is well defined. If two elements $E_1^e$ and $E_2^e$ are neighbours and share one common side $e$, then there are two traces of $v$ belonging to $e$. We assume that the normal vector $\bm n_e$ is oriented from $E_1^e$ to $E_2^e$. Then the average and jump are defined, respectively, by
$$\{v\}=\frac{1}{2}(v|_{\partial E_1^e}+v|_{\partial E_2^e}),\ \ \llbracket v\rrbracket=(v|_{\partial E_1^e}-v|_{\partial E_2^e}), \ \forall e\in\partial E_1^e\bigcap \partial E_2^e.$$
If $e$ is on $\partial\Omega$, we have
$$\{v\}=\llbracket v\rrbracket=v|_{e},~~\forall e\in\partial E\bigcap\partial\Omega.$$

\subsection{HDG scheme}
For designing the DG method of fractional derivative, we rewrite (\ref{eq1}) as a low order system (see \cite{r9,r10}).
Firstly, we introduce two auxiliary variables $\bm p,\bm\sigma$ and set
\begin{align*}
\centering
\begin{cases}
\bm p=\nabla u,
\\
\bm\sigma=I_{\bm x}^{\bar{\bm\alpha}}\bm p=(_{a}\!I_x^{\alpha_1}p_x,
_{c}\!I_{y}^{\alpha_2}p_y).
 \end{cases}
 \end{align*}
As Ref. \cite{r2}, let $\psi(\bm x,t)=(1+|\bm b(\bm x,t)|^2)^{\frac{1}{2}}$, where $|\bm b(\bm x,t)|^{2}=b_1^2+b_2^2$. Hence, the characteristic direction associated with $\partial_t u+\bm b\cdot\nabla u$ is denoted by $\partial_\tau=\frac{\partial_t}{\psi}+\frac{\bm b\cdot\nabla}{\psi}$. Then, from (\ref{equation2.4}) and (\ref{equation2.5}), Eq. (\ref{eq1}) can be rewritten as a mixed form \cite{r4,r5,r6,r9,r10}:

\begin{align}\label{eqloworder}
\centering
\begin{cases}
\psi\partial_\tau u-\nabla\cdot\bm\sigma=f, & (x,y,t)\in \Omega\times J,\cr
\bm\sigma-I_{\bm x}^{\bar{\bm\alpha}}\bm p=0, & (x,y,t)\in \Omega\times J,\cr
\bm p-\nabla u=0, & ( x,y,t)\in\Omega\times J,\cr
u(x,y,0)=u_{0}(x,y), & (x,y)\in\Omega,\cr
u(x,y,t)=0, & (x,y,t)\in\partial\Omega\times J,
  \end{cases}
  \end{align}
where $I_{\bm x}^{ \bar{\bm\alpha}}$ is defined in (\ref{eq2.6}).

For an arbitrary subset $E\in\mathscr{E}_h$, we multiply the first, second, and the third equation of (\ref{eqloworder}) by the smooth test functions $v,\bm\omega$, and $\bm q$, respectively. In order to obtain a symmetric weak variational formulation, we only integrate the first equation of (\ref{eqloworder}) by parts, and obtain
\begin{align}
\centering
 \begin{cases}
  \int_E\psi\partial_\tau u vd\bm x
 +\int_E\bm\sigma\cdot\nabla vd\bm x-\int_{\partial E}\bm\sigma\cdot\bm n_Evds=\int_E fvd\bm x,\cr
 \cr
 \int_E\bm\sigma\cdot\bm\omega d\bm x-\int_EI_{\bm x}^{\bar{\bm\alpha}}\bm p\cdot\bm\omega d\bm x=0,\cr
 \cr
 \int_{E}\bm p\cdot\bm q d\bm x-\int_E\nabla u\cdot\bm qd\bm x=0,
 \end{cases}
  \end{align}
where $\bm n_E$ is the outward unit normal to $\partial E$. Note that the above equations are well defined for the functions $(u,\bm\sigma,\bm p)$ and $(v,\bm\omega,\bm q)$ in $\mathbb{V}\times\mathbb{Q}\times\mathbb{Q}$, where
\begin{align*}
\mathbb{V}=&\Big\{u\in L^2(\Omega): u|_{E}\in H^1(E), \ \forall E\in\mathscr{E}_h\Big\},
  \\
\mathbb{Q}=&\Big\{\bm p\in (L^2(\Omega))^2: \bm p|_{E}\in (H^1(E))^2, \ \forall E\in\mathscr{E}_h\Big\}.
\end{align*}

Next we will approximate the exact solution $(u,\bm\sigma,\bm p)$ with the functions $(u_h,\bm\sigma_h,\bm p_h)$ in the finite element spaces $\mathbb{V}_h\times\mathbb{Q}_h\times\mathbb{Q}_h\subset \mathbb{V}\times\mathbb{Q}\times\mathbb{Q}$, where

\begin{align*}
\mathbb{V}_h=&\Big\{u_h\in L^2(\Omega): u_h|_{E}\in P^k(E),\ \ \forall E\in\mathscr{E}_h\Big\},
  \\
\mathbb{Q}_h=&\Big\{\bm p_h\in (L^2(\Omega))^2: \bm p_h|_{E}\in (P^k(E))^2,\ \ \forall E\in\mathscr{E}_h\Big\},
\end{align*}
where the finite element space $P^k(E)$ denotes the set of polynomials of degree less than or equal to $k\geq 0$.

Thus, the approximate solution $(u_h,\bm\sigma_h,\bm p_h)$ satisfies the weak formulation, for all $(v,\bm\omega,\bm q)\in \mathbb{V}_h\times\mathbb{Q}_h\times\mathbb{Q}_h$,
  \begin{align}\label{eq3.3}
\centering
 \begin{cases}
  \int_E\psi\partial_\tau u_hvd\bm x
 +\int_E\bm\sigma_h\cdot\nabla vd\bm x
 -\int_{\partial E}\bm\sigma_h^\star\cdot\bm n_Evds=\int_E fvd\bm x,\cr
 \cr
 \int_E\bm\sigma_h\cdot\bm\omega d\bm x-\int_EI_{\bm x}^{\bar{\bm\alpha}}\bm p_{h}\cdot\bm\omega d\bm x=0,\cr
 \cr
 \int_{E}\bm p_h\cdot\bm q d\bm x-\int_{E}\nabla u_h\cdot\bm qd\bm x=0,
 \end{cases}
  \end{align}
where the numerical fluxes are well chosen as
$\bm\sigma_h^\star=\{\bm\sigma_h\},\ \forall e\in\mathscr{E}_h^B$ in order to ensure the stability of the scheme and its accuracy.

It is well known that the fluxes $\bm\sigma_h^\star=\{\bm\sigma_h\}$ are consistent. Inspired by the penalty Galerkin methods
\cite{r17} and noting the fact that $\llbracket u\rrbracket\big|_{e}=0,\forall e\in\mathscr{E}_h^B$
and $\llbracket\bm\sigma\rrbracket=0,\forall e\in\mathscr{E}_h^i$, a symmetric and stable DG scheme is derived as follows.
Substituting the flux $\bm\sigma_h^\star=\{\bm\sigma_h\}$ into (\ref{eq3.3}), summing over all the
elements, and adding the penalty terms, we observe that for $(u_h,\bm\sigma_h,\bm p_h)\in \mathbb{V}_h\times\mathbb{Q}_h\times\mathbb{Q}_h$, the semi-discrete variational formulation is given by
  \begin{align}\label{eq3.4}
\begin{cases}
(\psi\partial_\tau u_h,v)+(\bm\sigma_h,\nabla v)
-(\{\bm\sigma_h\}\cdot\bm n_e,\llbracket v\rrbracket)_{\mathscr{E}_h^B}
+\epsilon_{1}(\llbracket u_h\rrbracket,\llbracket v\rrbracket)_{\mathscr{E}_h^B}=(f,v),\cr
 \\
(\bm\sigma_h,\bm\tau)-(I_{\bm x}^{\bar{\bm\alpha}}\bm p_{h},\bm\omega)=0,
 \cr
 \\
(\bm p_h,\bm q)-(\nabla u_h,\bm q)+(\llbracket u_h\rrbracket,\{\bm q\}\cdot\bm n_e)_{\mathscr{E}_h^B}+\epsilon_{2}(\llbracket\bm\sigma_h\rrbracket,\llbracket \bm q\rrbracket)_{\mathscr{E}_h^i}
 =0.
 \end{cases}
\end{align}

For any $(v,\bm\omega,\bm q)\in\mathbb{V}_h\times\mathbb{Q}_h\times\mathbb{Q}_h$, the exact solution of (\ref{eq1}) is expected to be continuously differentiable with respect to the variables $x$ and $y$, which keeps the consistency of the scheme. The term $(\llbracket u\rrbracket,\{\bm q\}\cdot\bm n_e)_{\mathscr{E}_h^B}$ vanishes since the exact solution $u$ satisfies $\llbracket u\rrbracket\big|_{e}=0,\forall e\in\mathscr{E}_h^B$. Note that $\epsilon_1(\llbracket u\rrbracket,\llbracket v\rrbracket)_{\mathscr{E}_h^B}$ penalizes the jump of the function $u$, whereas $\epsilon_2(\llbracket\bm\sigma\rrbracket,\llbracket \bm q\rrbracket)_{\mathscr{E}_h^i}$ penalizes the jump of the function $\bm\sigma$. Here $\epsilon_1$ and $\epsilon_2$ are the positive constants to be chosen. Unfortunately the third equation of (\ref{eq3.4}) makes the DG method lose its locality, since $\bm p_h$ is a function of $u_h$ and $\bm\sigma_h$, $\bm p_h$ can not be eliminated from the third equation. So we have to simultaneously solve the three unknowns $u_h,p_{xh},p_{yh}$. Although the extra unknowns can not be eliminated in the HDG methods, our choice of fluxes makes the error analysis available.
Above and throughout, the following notations are used,
\begin{align*}
(w,v)=\sum_{E\in\mathscr{E}_h}(w,v)_E,~(w,v)_{\mathscr{E}_h^i}=\sum_{e\in\mathscr{E}_h^i}(w,v)_e,~ (w,v)_{\mathscr{E}_h^B}=\sum_{e\in\mathscr{E}_h^B}(w,v)_e.
\end{align*}

\subsection{Dealing with time}
After performing the HDG approximation, we discretize the time derivative with the characteristic method.
For the given positive integer $N$, let $0=t^0<t^1<\cdots<t^N=T$ be a partition of $J$ into subintervals $J^n=(t^{n-1},t^n]$ with uniform mesh and the interval length $\Delta t=t^n-t^{n-1}, 1\leq n\leq N$. The characteristic tracing back along the field $\bm b$ of a point $\bm x=(x,y)\in\Omega$ at time $t^n$ to $t^{n-1}$ is approximated by
\cite{r2,r3,r1}
$$\check{\bm x}(\bm x,t^{n-1})=\bm x-\bm b(\bm x,t^n)\Delta t. $$

Therefore, the approximation for the hyperbolic part of (\ref{eq1}) at time $t^n$ can be approximated as
$$\psi^n\partial_\tau u^n\approx\frac{u^n-\check{u}^{n-1}}{\Delta t},$$
where $u^n=u(\bm x,t^n)$, $\check{u}^{n-1}=u(\check{\bm x}(\bm x,t^{n-1}),t^{n-1})$, and $\check{u}^0=u^0(\bm x)$.

\begin{rem} [see \cite{r1}]\label{Remark 2}
Assume that the solution $u$ of (\ref{eq1}) is sufficiently regular. Under the assumption of the function $\bm b$, we have
\begin{align*}
\left\|\psi^n\partial_\tau u^n-\frac{u^n-\check{u}^{n-1}}{\Delta t}\right\|^2_{L^2(\Omega)}\leq C\parallel\psi^{(4)}\parallel_{L^\infty(J;L^\infty(\Omega))}\parallel\partial_{\tau\tau}u\parallel^2_{L^2({J}^n;L^2(\Omega))}\Delta t.
\end{align*}
\end{rem}
Thus, the fully discrete scheme corresponding to the variational formulation (\ref{eq3.4}) is
to find $(u_h^n,\bm\sigma_h^n,\bm p_h^n)\in \mathbb{V}_h\times\mathbb{Q}_h\times\mathbb{Q}_h$, for any
$(v,\bm\tau,\bm q)\in \mathbb{V}_h\times\mathbb{Q}_h\times\mathbb{Q}_h$, such that

\begin{align}\label{eq3.5}
\begin{cases}
\big(\frac{u_h^n-\check{u}^{n-1}_h}{\Delta t},v\big)
+(\bm\sigma_h^n,\nabla v)-(\{\bm\sigma^n_h\}\cdot\bm n_e,\llbracket v\rrbracket)_{\mathscr{E}_h^B}
+\epsilon_{1}(\llbracket u_h^n\rrbracket,\llbracket v\rrbracket)_{\mathscr{E}_h^B}=(f^n,v),\cr
 \\
 (\bm\sigma_h^n,\bm\omega)-(I_{\bm x}^{\bar{\bm\alpha}}\bm p_{h}^n,\bm\omega)=0,
 \cr
 \\
 (\bm p^n_h,\bm q)-(\nabla u_h^n,\bm q)+(\llbracket u^n_h\rrbracket,\{\bm q\}\cdot\bm n_e)_{\mathscr{E}_h^B}+\epsilon_{2}(\llbracket\bm\sigma^n_h\rrbracket,\llbracket \bm q\rrbracket)_{\mathscr{E}_h^i}
 =0,
 \end{cases}
\end{align}
where $\check{u}_h^{n-1}=u_h(\check{\bm x}(\bm x,t^{n-1}),t^{n-1}),\check{u}_h^0=u^0$.

Define the bilinear forms by
\begin{align*}
\mathbbm{a}(\bm\sigma_h^n,v)&:=(\bm\sigma_h^n,\nabla v)-(\{\bm\sigma^n_h\}\cdot\bm n_e,\llbracket v\rrbracket)_{\mathscr{E}_h^B},\ \
\mathbbm{c}(\bm p_h^n,\bm q):=(\bm p_h^n,\bm q),
\\
\mathbbm{d}(u_h^n,v)&:=\epsilon_1(\llbracket u_h^n\rrbracket,\llbracket v\rrbracket)_{\mathscr{E}_h^B},\ \ \
\mathbbm{e}(\bm\sigma_h^n,\bm q):=\epsilon_2\big(\llbracket\bm\sigma_h^n\rrbracket,
\llbracket
\bm q\rrbracket)_{\mathscr{E}_h^i},
\end{align*}
and the linear form
\begin{align*}
\mathcal {F}(v)&:=(f^n,v)~~~\forall v\in \mathbb{V}_h.
\end{align*}
We can rewrite (\ref{eq3.5}) as a compact formulation: Find $(u_h^n,\bm\sigma_h^n,\bm p_h^n)\in \mathbb{V}_h\times\mathbb{Q}_h\times\mathbb{Q}_h$ at time $t=t^n$, such that
\begin{align}\label{eq3.6}
\begin{cases}
\big(\frac{u_h^n-\check{u}^{n-1}_h}{\Delta t},v\big)+\mathbbm{a}(\bm\sigma_h^n,v)+\mathbbm{d}(u_h^n,v)=\mathcal {F}(v), & \forall v\in \mathbb{V}_h,
\cr
\\
\mathbbm{c}(\bm\sigma_h^n,\bm\omega)-\mathbbm{c}(I_{\bm x}^{\bar{\bm\alpha}}\bm p_{h}^n,\bm\omega)=0, & \forall\bm\omega\in \mathbb{Q}_h,
\cr
\\
\mathbbm{c}(\bm p_h^n,\bm q)-\mathbbm{a}(\bm q,u_h^n)+\mathbbm{e}(\bm\sigma_h^n,\bm q)=0, & \forall \bm q\in\mathbb{Q}_h.
\end{cases}
\end{align}

%
%

\section{Stability analysis and error estimate}
This section focuses on providing the proof of the unconditional stability and the error estimates of the schemes.
\subsection{Stability analysis}

In the following, $C$ indicates a generic constant independent of $h$ and $\Delta t$, which takes different values in different occurrences.
\begin{lem} [\cite{r2}]\label{lemma4.1}
  If $\bm b\in L^\infty(J;W^{1,\infty}(\Omega)^2)$, for any function $v\in L^2(\Omega)$ and each $n$, there is
\begin{align}
\parallel\check{v}\parallel^2_{L^2(\Omega)}-\parallel v\parallel^2_{L^2(\Omega)}\leq C\Delta t\parallel v\parallel^2_{L^2(\Omega)},
\end{align}
where $\check{v}(\bm x)=v(\bm x-\bm b(\bm x,t^n)\Delta t)$.
\end{lem}

\begin{thm}[Numerical stability]\label{theorem4.2}
If $\bm b\in L^\infty(J;W^{1,\infty}(\Omega)^2)$, the HDG scheme (\ref{eq3.5}) is stable, i.e., for any integer
$N=1,2,\cdots$, there is
\begin{align}
\begin{split}
\parallel u_h^N\parallel_{L^2(\Omega)}^2+2\Delta t\sum_{n=1}^N\big|(u_h^n,\bm\sigma_h^n,\bm p_h^n)\big|^{2}_{\mathcal{A}}\leq C\Delta t\sum_{n=1}^N\parallel f^n\parallel_{L^2(\Omega)}^{2}+C\parallel u^{0}\parallel_{L^2(\Omega)}^2,
\end{split}
\end{align}
where $u_h^{0}=u^{0}$, and the semi-norm $|\cdot|_{\mathcal{A}}$ is defined as
\begin{align}\label{eq3.7}
\begin{split}
&\big|(u_h^n,\bm\sigma_{h}^n,\bm p_{h}^n)\big|^2_{\mathcal {A}}\\
&=\mathbbm{d}(u_h^n,u_h^n)+\mathbbm{c}(I_{\bm x}^{\bar{\bm\alpha}}\bm p_{h}^n,\bm p_h^n)+\mathbbm{e}(\bm\sigma_h^n,\bm\sigma_h^n)
\\
&=\cos(\alpha_1\pi/2)\int_c^d \parallel p_{xh}^n(\cdot,y)\parallel^2_{J_{R,0}^{-\alpha_1/2}(a,b)}dy
+\epsilon_{1}\sum_{e\in\mathscr{E}_h^B}
\parallel\llbracket u_h^n\rrbracket\parallel_{L^2(e)}^2
\\&~~~~+\cos(\alpha_2\pi/2)\int_a^b \parallel p_{yh}^n(x,\cdot)\parallel^2_{J_{R,0}^{-\alpha_2/2}(c,d)}dx
+\epsilon_{2}\sum_{e\in\mathscr{E}_h^i}\parallel
\llbracket\bm\sigma_h^n\rrbracket\parallel_{L^2(e)}^2.
\end{split}
\end{align}
\end{thm}

\begin{proof}
Let $v=2\Delta t u_h^n,\,\bm\omega=-2\Delta t\bm p_h^n,\,\bm q=2\Delta t\bm \sigma_h^n$ in the equations of (\ref{eq3.6}), respectively. By the symmetry of the bilinear formulas, adding the above equations, we obtain
\begin{align*}
2\Delta t\mathcal{F}(u_h^n)&=2\Delta t\mathbbm{c}(I_{\bm x}^{\bar{\bm\alpha}}\bm p_h^n,\bm p_h^n)+2\Delta t\mathbbm{e}(\bm\sigma_h^n,\bm\sigma_h^n)
+2\big(u_h^n-\check{u}^{n-1}_h,u_h^n\big)+2\Delta t\mathbbm{d}(u_h^n,u_h^n).
\end{align*}
Following from
\begin{align*}
2\big(u_h^n-\check{u}_h^{n-1},u_h^n\big)\geq\parallel u_h^n\parallel_{L^2(\Omega)}^2-\parallel\check{u}_h^{{n-1}}\parallel_{L^2(\Omega)}^2,
\end{align*}
the Young inequality, the definition of $\mathcal{F}$ and $|\cdot|_{\mathcal{A}}$, and Lemma \ref{lemma4.1}, we have
\begin{align*}
&\parallel u_h^n\parallel_{L^2(\Omega)}^2-\parallel u_h^{{n-1}}\parallel_{L^2(\Omega)}^2+2\Delta t\big|(u_h^n,\bm\sigma_h^n,\bm p_h^n)\big|^2_{\mathcal{A}}
\\
&\leq C\Delta t\parallel u_h^{n-1}\parallel_{L^2(\Omega)}^2
+\Delta t\big(\parallel u_h^n\parallel_{L^2(\Omega)}^2+\parallel f^n\parallel_{L^2(\Omega)}^2\big).
\end{align*}
Summing from n=1,2,...,N, we get
\begin{align*}
&\parallel u_h^N\parallel_{L^2(\Omega)}^2+2\Delta t\sum_{n=1}^N\big|(u_h^n,\bm\sigma_h^n,\bm p_h^n)\big|^2_{\mathcal {A}}
\\
&\leq C\Delta t\sum_{n=1}^N\parallel u_h^{n}\parallel_{L^2(\Omega)}^2+(1+C\Delta t)\parallel u_h^{0}\parallel_{L^2(\Omega)}^2+\Delta t\sum_{n=1}^N\parallel f^n\parallel_{L^2(\Omega)}^2.
\end{align*}
Using the discrete Gr\"{o}nwall inequality, with $C\Delta t<1, \ \forall\, N\geq 1$, there is
\begin{equation}
\begin{split}
&\parallel u_h^N\parallel_{L^2(\Omega)}^2+
2\Delta t\sum_{n=1}^N\big|(u_h^n,\bm\sigma_h^n,\bm p_h^n)\big|^2_{\mathcal{A}}
\leq C\parallel u_h^{0}\parallel_{L^2(\Omega)}^2+C\Delta t\sum_{n=1}^N\parallel f^n\parallel_{L^2(\Omega)}^2.
\end{split}
\end{equation}
\end{proof}

\subsection{Error estimates}
In this subsection we state and discuss the error bounds for the HDG scheme. The main steps of our error analysis follow the classical methods in finite element analysis, i.e., the so-called Galerkin orthogonality property.
As usual, we denote the errors $(e_u^n,\bm e_{\bm\sigma}^n,\bm e_{\bm p}^n)=(u^n-u_h^n,\bm\sigma^n-\bm \sigma^n_h,\bm p^n-\bm p^n_h)$ by
\begin{align*}
\begin{split}
(e_u^n,\bm e_{\bm\sigma}^n,\bm e_{\bm p}^n)=(u^n-\Pi u^n,\bm\sigma^n-\bm\Pi\bm\sigma^n,\bm p^n-\bm\Pi\bm p^n)+(\Pi e_u^n,\bm\Pi
\bm e_{\bm\sigma}^n,\bm\Pi\bm e_{\bm p}^n),
\end{split}
\end{align*}
where $\Pi$ and $\bm\Pi=(\Pi,\Pi)$ are the $L^2$-projection and $(L^2)^2$-projection operators from $\mathbb{V}$ and $\mathbb{Q}$ onto the finite element spaces $\mathbb{V}_h$ and $\mathbb{Q}_h$, respectively.
From (\ref{eq3.6}), we obtain the compact form
\begin{equation}
\Big(\frac{u_h^n-\check{u}_h^{n-1}}{\Delta t},v\Big)+\mathcal {A}(u_h^n,\bm \sigma_{h}^n, \bm p_{h}^n;v,\bm \omega,\bm q)=\mathcal{F}(v),
\end{equation}
where
\begin{equation}
\begin{split}
&~~~\mathcal{A}(u_h^n,\bm \sigma_{h}^n,\bm p_{h}^n;v,\bm\omega,\bm q)
\\
&=\mathbbm{a}(\bm\sigma_h^n,v)+\mathbbm{d}(u_h^n,v)
+\mathbbm{c}(\bm\sigma_h^n,\bm\omega)-\mathbbm{c}(I_{\bm x}^{\bar{\bm\alpha}}\bm p_h^n,\bm\omega)
+\mathbbm{c}(\bm p_h^n,\bm q)-\mathbbm{a}(\bm q,u_h^n)+\mathbbm{e}(\bm\sigma_h^n,\bm q).
\end{split}
\end{equation}

\begin{lem}
Assume that the solution $u$ of problem (\ref{eq1}) is sufficiently regular. Then
\begin{align}\label{erro}
\begin{split}
&~~~\big(\psi^n\partial_\tau u^n-\frac{u_h^n-\check{u}_h^{n-1}}{\Delta t},\Pi e_u^n\big)
+\big|(\Pi e_u^n,\bm\Pi\bm e_{\bm\sigma}^n,\bm\Pi\bm e_{\bm p}^n)\big|^2_{\mathcal{A}}
\\
&=\mathcal{A}\big(\Pi u^n-u^n,\bm\Pi\bm\sigma^n-\bm\sigma^n,\bm\Pi\bm p^n-\bm p^n;\Pi e_u^n,-\bm\Pi \bm e_{\bm p}^n,\bm\Pi\bm e_{\bm\sigma}^n\big).
\end{split}
\end{align}
\end{lem}

\begin{proof}
By the consistency of the numerical fluxes,
the exact solution $(u,\bm\sigma,\bm p)$ satisfies (\ref{eq3.4}). Taking $v=\Pi e_u^n,\bm\omega=-\bm\Pi\bm e_{\bm p}^n,\bm q=\bm\Pi\bm e_{\bm\sigma}^n$ and subtracting (\ref{eq3.5}) from (\ref{eq3.4}) yield
\begin{align}\label{erro1}
\begin{split}
\big(\psi^n\partial_\tau u^n-\frac{u_h^n-\check{u}_h^{n-1}}{\Delta t},\Pi e_u^n\big)
+\mathcal{A}\big(e_u^n,\bm e_{\bm\sigma}^n,\bm e_{\bm p}^n;\Pi e_u^n,-\bm\Pi\bm e_{\bm p}^n,\bm\Pi\bm e_{\bm\sigma}^n\big)=0
\end{split}
\end{align}
and
\begin{align}\label{erro2}
\begin{split}
\big|(\Pi e_u^n,\bm\Pi\bm e_{\bm\sigma}^n,\bm\Pi\bm e_{\bm p}^n)\big|^2_{\mathcal {A}}
=\mathcal{A}\big(\Pi e_u^n,\bm\Pi\bm e_{\bm\sigma}^n,\bm\Pi\bm e_{\bm p}^n;\Pi e_u^n,-\bm\Pi\bm e_{\bm p}^n,\bm\Pi\bm e_{\bm\sigma}^n\big).
\end{split}
\end{align}
By the Galerkin orthogonality, there is
\begin{align}\label{erro3}
\begin{split}
&~~~\mathcal{A}\big(e_u^n,\bm e_{\bm\sigma}^n,\bm e_{\bm p}^n;\Pi e_u^n,-\bm\Pi\bm e_{\bm p}^n,\bm\Pi\bm e_{\bm\sigma}^n\big)\\
&=\mathcal{A}\big(\Pi e_u^n,\bm\Pi\bm e_{\bm\sigma}^n,\bm\Pi\bm e_{\bm p}^n;\Pi e_u^n,-\bm\Pi\bm e_{\bm p}^n,\bm\Pi\bm e_{\bm\sigma}^n\big)
\\
&-\mathcal{A}\big(\Pi u^n-u^n,\bm\Pi\bm\sigma^n-\bm\sigma^n,\bm\Pi\bm p^n-\bm p^n;\Pi e_u^n,-\bm\Pi \bm e_{\bm p}^n,\bm\Pi\bm e_{\bm\sigma}^n\big).
\end{split}
\end{align}
Substituting the equalities (\ref{erro2}) and (\ref{erro3}) into (\ref{erro1}) leads to the desired result.
\end{proof}
%

Next we review two lemmas for our analysis. The first one is the standard approximation result for the  $L^2$-projection operator $\Pi$ from $H^{s+1}(E)$ onto $V_h(E)=\{v;v\big|_E\in P^k(E)\}$ satisfying $\Pi v=v$ for any $v\in P^k(E)$. The second one is the standard trace inequality.
\begin{lem} [\cite{r4}]\label{lemma4.3}
 Let $v\in H^{s+1}(E),\,s\geq 0$. $\Pi$ is the $L^2$-projection operator from $H^{s+1}(E)$ onto $V_h(E)$ such that $\Pi v=v$ for any $v\in P^k(E)$. Then, for $m=0,1,$
\begin{align}
\begin{cases}
\big|v-\Pi v\big|_{H^m(E)}\leq C h_{E}^{min\{s,k\}+1-m}\parallel v\parallel_{H^{s+1}(E)},
\\
\parallel v-\Pi v\parallel_{L^2(\partial E)}\leq C h_{E}^{min\{s,k\}+\frac{1}{2}}\parallel v\parallel_{H^{s+1}(E)}.
\end{cases}
\end{align}
\end{lem}

\begin{lem}[\cite{r4}]\label{lemma4.4}
 There exists a generic  constant $C$ being independent of $h_E$, for any $v\in V_h(E)$, such that
\begin{equation}
\parallel v\parallel_{L^2(\partial E)}\leq Ch_E^{-\frac{1}{2}}\parallel v\parallel_{L^2(E)}.
\end{equation}
\end{lem}
Now we are ready to prove our main results.
\subsubsection{The characteristic term}

In this subsection, we estimate the first left-side term of (\ref{erro}).
\begin{lem}[\cite{r2}]\label{lemma4.5}
If $\bm b\in L^\infty(J;W^{1,\infty}(\Omega)^2)$, for any function $ v\in H^1(\Omega)$ and each $n$,
\begin{align}
\parallel v-\check{v}\parallel_{L^2(\Omega)}\leq C\Delta t\parallel\nabla v\parallel_{L^2(\Omega)},
\end{align}
where $\check{v}=v(\check{\bm x})=v(\bm x-\bm b^n\Delta t)$.
\end{lem}

The following result is a straightforward consequence of the estimate of the first left-side term of (\ref{erro}).
\begin{thm}\label{theorem4.6}
Assume that the solution $u$ of problem (\ref{eq1}) is sufficiently smooth and $u_h^n$ satisfies (\ref{eq3.5}). If $\bm b\in L^\infty(J;W^{1,\infty}(\Omega)^2)$, we have
\begin{equation}\label{eq4.15}
\begin{split}
&\left(\psi^n\partial_\tau u^n-\frac{u_h^n-\check{u}_h^{n-1}}{\Delta t},\Pi e_u^n\right)
\\
&\geq \frac{1}{2\Delta t}\big(\parallel\Pi e_u^n\parallel_{L^2(\Omega)}^2-\parallel\Pi e_u^{n-1}\parallel_{L^2(\Omega)}^2\big)-C\parallel\Pi e_u^{n-1}\parallel_{L^2(\Omega)}^2
\\
&~~~~-C\Delta t\parallel\partial_{\tau\tau}u\parallel_{L^2(J^n;L^2(\Omega))}^2
-\frac{C}{\Delta t}\parallel\partial_t(\Pi
u-u)\parallel_{L^2(J^n;L^2(\Omega))}^2
\\
&~~~~ -C\parallel\nabla(\Pi u^{n-1}-u^{n-1})\parallel_{L^2(\Omega)}^2
-C\parallel\Pi e_u^n\parallel_{L^2(\Omega)}^2.
\end{split}
\end{equation}
\end{thm}
\begin{proof}
From (\ref{erro}), it can be noted that
\begin{align}\label{eq4.16}
\begin{split}
&\left(\psi^n\partial_\tau u^n-\frac{u_h^n-\check{u}_h^{n-1}}{\Delta t},\Pi e_u^n\right)
\\
&=\left(\frac{\Pi e_u^n-\Pi\check{e}_u^{n-1}}{\Delta t},\Pi e_u^n\right)
+\left(\psi^n\partial_\tau u^n-\frac{u^n-\check{u}^{n-1}}{\Delta t},\Pi e_u^n\right)
\\
&~~~~-\left(\frac{(\Pi u^n-u^n)-(\Pi\check{u}^{n-1}-\check{u}^{n-1})}{\Delta t},\Pi e_u^n\right)
\\
&=\sum_{i=1}^3\mathcal {B}_i.
\end{split}
\end{align}
Using Lemma \ref{lemma4.1}, we obtain
\begin{align*}
\mathcal{B}_1&=\big(\frac{\Pi e_u^n-\Pi \check{e}_u^{n-1}}{\Delta t},\Pi e_u^n\big)\\
&=\frac{1}{2\Delta t}\big(\parallel\Pi e_u^n\parallel_{L^2(\Omega)}^2-\parallel\Pi\check{e}_u^{n-1}\parallel_{L^2(\Omega)}^2
+\parallel\Pi e_u^n-\Pi\check{e}^{n-1}_u\parallel_{L^2(\Omega)}^2\big)
\\
&\geq\frac{1}{2\Delta t}\big(\parallel\Pi e_u^n\parallel_{L^2(\Omega)}^2-\parallel\Pi\check{e}_u^{n-1}\parallel_{L^2(\Omega)}^2\big)
\\
&\geq\frac{1}{2\Delta t}\big(\parallel\Pi e_u^n\parallel_{L^2(\Omega)}^2-\parallel\Pi e_u^{n-1}\parallel_{L^2(\Omega)}^2\big)-C\parallel\Pi e_u^{n-1}\parallel_{L^2(\Omega)}^2,
\end{align*}
where $\Pi\check{e}_u^{n-1}=\Pi\check{u}^{n-1}-\check{u}_h^{n-1}$. Also by the Taylor expansion and the H\"{o}lder inequality, there are
\begin{align*}
\mid\mathcal{B}_2\mid&=\left|\left(\psi^n\partial_\tau u^n-\frac{u^n-\check{u}^{n-1}}{\Delta t},\Pi e_u^n\right)\right|
\\
&\leq C\Delta t\parallel\partial_{\tau\tau}u\parallel^2_{L^2({J}^n;L^2(\Omega))}+C\parallel\Pi e_u^n\parallel^2_{L^2(\Omega)}
\end{align*}
and
\begin{align*}
-\mathcal{B}_3&=\Big(\frac{(\Pi u^n-u^n)-(\Pi\check{u}^{n-1}-\check{u}^{n-1})}{\Delta t},\Pi e_u^n\Big)
\\
&=\Big(\frac{(\Pi u^n-u^n)-(\Pi u^{n-1}-u^{n-1})}{\Delta t},\Pi e_u^n\Big)
\\
&+\Big(\frac{(\Pi u^{n-1}-u^{n-1})-(\Pi\check{u}^{n-1}-\check{u}^{n-1})}{\Delta t},\Pi e_u^n\Big)
\\
&=\mathcal{S}_1+\mathcal {S}_2,
\end{align*}
where
\begin{align*}
\mathcal{S}_1&=\Big(\frac{(\Pi u^n-u^n)-(\Pi u^{n-1}-u^{n-1})}{\Delta t},\Pi e_u^n\Big)
\\
&\leq\frac{1}{\Delta t}\parallel\Pi e_u^n\parallel_{L^2(\Omega)}\int_{t^{n-1}}^{t^n}\parallel
\partial_t(\Pi u-u)\parallel_{L^2(\Omega)}dt
\\
&\leq C\parallel \Pi e_u^n\parallel_{L^2(\Omega)}^2+\frac{C}{\Delta t}\parallel\partial_t(\Pi
u-u)\parallel_{L^2(J^n;L^2(\Omega))}^2,
\end{align*}
and
\begin{align*}
\mathcal{S}_2&=\Big(\frac{(\Pi u^{n-1}-u^{n-1})-(\Pi\check{u}^{n-1}-\check{u}^{n-1})}{\Delta t},\Pi e_u^n\Big)
\\
&\leq C\parallel\Pi e^n_u\parallel_{L^2(\Omega)}^2+C\parallel\nabla(\Pi u^{n-1}-u^{n-1})\parallel_{L^2(\Omega)}^2,
\end{align*}
follow from Cauchy-Schwarz's inequality, Young's inequality and Lemma~\ref{lemma4.5}.
Substituting  $\mathcal{B}_1,\mathcal{B}_2,\mathcal{B}_3$ into (\ref{eq4.16}), the desired result is reached.
\end{proof}

\subsubsection{The right-hand side term}

In this subsection, we use the general analytic methods to get the bound of the right side term of (\ref{erro}).
\begin{thm}\label{theorem4.7}
Let $u$ be sufficiently smooth solution of (\ref{eqloworder}). $(\Pi u^n,\bm\Pi\bm\sigma^n
,\bm\Pi\bm p^n)$ are standard $L^2$-projection operators of $(u^n,\bm\sigma^n,\bm p^n)$, and $(u_h^n,\bm\sigma_h^n,\bm p_h^n)$ solve (\ref{eq3.5}). If $\bm b\in L^\infty(J;W^{1,\infty}(\Omega)^2)$, we have
\begin{align}
\begin{split}
&~~~\big|\mathcal{A}\big(\Pi u^n-u^n,\bm\Pi\bm\sigma^n-\bm\sigma^n,\bm\Pi\bm p^n-\bm p^n;
\Pi e_u^n,-\bm\Pi\bm e_{\bm p}^n,\bm\Pi\bm e_{\bm\sigma}^n\big)\big|
\\
&\leq C\epsilon_{\alpha_1}\int_c^d \parallel\Pi e_{p_x}^n(\cdot,y)\parallel^2_{J_{R,0}^{-\alpha_1/2}(a,b)}dy
+\left(\frac{C}{\epsilon_{1}}+C\epsilon_{1}\right)h^{2k+1}+\frac{C}{\epsilon_{\alpha_1}}h^{2k+2}
\\
&
~~~~ +C\epsilon_{\alpha_2}\int_a^b \parallel\Pi e_{p_y}^n(x,\cdot)\parallel^2_{ J_{R,0}^{-\alpha_2/2}(c,d)}dx+\left(\frac{C}{\epsilon_{2}}+C\epsilon_{2}\right)h^{2k+1}+\frac{C}{\epsilon_{\alpha_2}}h^{2k+2}
\\
&~~~~ +\frac{\epsilon_{1}}{2}\sum_{e\in\mathscr{E}_h^B}
\parallel\llbracket\Pi e_u^n\rrbracket\parallel_{L^2(e)}^2+\frac{\epsilon_{2}}{2}\sum_{e\in\mathscr{E}_h^i}\parallel
\llbracket\bm\Pi\bm e_{\bm\sigma}^n\rrbracket\parallel_{L^2(e)}^2.
\end{split}
\end{align}
\end{thm}

\begin{proof}
From the definition of $\mathcal{A}$, we have
\begin{align}\label{eq4.18}
\begin{split}
&\mathcal {A}\big(\Pi u^n-u^n,\bm\Pi\bm\sigma^n-\bm\sigma^n,\bm\Pi\bm p^n-\bm p^n;
\Pi e_u^n,-\bm\Pi \bm e_{\bm p}^n,\bm\Pi\bm e_{\bm\sigma}^n\big)
\\
&\leq\big|\mathbbm{a}(\bm\Pi\bm\sigma^n-\bm\sigma^n,\Pi e_u^n)\big|+\big|\mathbbm{c}(I_{\bm x}^{\bar{\bm\alpha}}(\bm\Pi \bm p^n-\bm p^n),\bm\Pi \bm e_{\bm p}^n)\big|
\\
&~~~~+\big|\mathbbm{a}(\bm\Pi \bm e_{\bm\sigma}^n,\Pi u^n-u^n)\big|
+\big|\mathbbm{d}(\Pi u^n-u^n,\Pi e_u^n)\big|
\\
&~~~~+\big|\mathbbm{e}(\bm\Pi\bm\sigma^n-\bm\sigma^n,\bm\Pi \bm e_{\bm\sigma}^n)\big|+\big|\mathbbm{c}(\bm\Pi\bm\sigma^n-\bm\sigma^n,-\bm\Pi\bm e_{\bm p}^n)\big|
\\
&~~~~+\big|\mathbbm{c}(\bm\Pi\bm p^n-\bm p^n,\bm\Pi\bm e_{\bm\sigma}^n)\big|
\\
&=\sum_{i=1}^7{T}_i.
\end{split}
\end{align}
Using H\"{o}lder's, Young's inequalities and Lemma~\ref{lemma4.3}, we obtain
\begin{align*}
{T}_1&=\big|\mathbbm{a}(\bm\Pi\bm\sigma^n-\bm\sigma^n,\Pi e_u^n)\big|
=\big|(\{\bm\Pi\bm\sigma^n-\bm\sigma^n\}\cdot\bm n_e,\llbracket\Pi e_u^n\rrbracket)_{\mathscr{E}_h^B}\big|
\\
&\leq\sum_{e\in\mathscr{E}_h^B}\parallel\{\bm\Pi\bm\sigma^n-\bm\sigma^n\}\cdot\bm n_e\parallel_{L^2(e)}\parallel\llbracket\Pi e_u^n\rrbracket\parallel_{L^2(e)}
\\
&\leq\sum_{e\in\mathscr{E}_h^B}\big(\frac{1}{\epsilon_1}\parallel\{\bm\Pi\bm\sigma^n-\bm\sigma^n\}\cdot\bm n_e\parallel_{L^2(e)}^2+\frac{\epsilon_1}{4}\parallel\llbracket\Pi e_u^n\rrbracket\parallel_{L^2(e)}^2\big)
\\
&\leq\frac{C}{\epsilon_{1}}h^{2k+1}+\frac{\epsilon_{1}}{4}\sum_{e\in\mathscr{E}_h^B}\parallel\llbracket\Pi e_u^n\rrbracket\parallel_{L^2(e)}^2.
\end{align*}
From Lemma~\ref{lemma2.3}, Lemma~\ref{lemma4.3}, Definition~\ref{definition2.7}, Definition~\ref{definition2.11}, and Theorem~\ref{theorem2.10}, it follows that
\begin{align*}
\begin{split}
{T}_2&=\big|\mathbbm{c}(I_{\bm x}^{\bar{\bm\alpha}}(\bm\Pi\bm p^n-\bm p^n),\bm\Pi\bm e_{\bm p}^n)\big|
\\
&=\big|(\Pi p_x^n- p_x^n,{_x\!}I^{\alpha_1}_b\Pi e_{p_x}^n)+(\Pi p_y^n- p_y^n,{_y\!}I^{\alpha_2}_d\Pi e_{p_y}^n)\big|
\\
&\leq\parallel\Pi p_x^n- p_x^n\parallel_{L^2(\Omega)}\left(\int_c^d \parallel\Pi e_{p_x}^n(\cdot,y)\parallel^2_{J_{R,0}^{-\alpha_1}(a,b)}dy\right)^{\frac{1}{2}}
\\
&~~~~+\parallel\Pi p_y^n- p_y^n\parallel_{L^2(\Omega)}\left(\int_a^b \parallel\Pi e_{p_y}^n(x,\cdot)\parallel^2_{J_{R,0}^{-\alpha_2}(c,d)}dx\right)^{\frac{1}{2}}
\\
&\leq C\parallel\Pi p_x^n- p_x^n\parallel_{L^2(\Omega)}\left(\int_c^d \parallel\Pi e_{p_x}^n(\cdot,y)\parallel^2_{J_{R,0}^{-\alpha_1/2}(a,b)}dy\right)^{\frac{1}{2}}
\\
&~~~~+C\parallel\Pi p_y^n- p_y^n\parallel_{L^2(\Omega)}\left(\int_a^b \parallel\Pi e_{p_y}^n(x,\cdot)\parallel^2_{J_{R,0}^{-\alpha_2/2}(c,d)}dx\right)^{\frac{1}{2}}
\\
&\leq \frac{C}{\epsilon_{\alpha_1}}h^{2k+2}+C\epsilon_{\alpha_1}\int_c^d
\parallel\Pi e_{p_x}^n(\cdot,y)\parallel^2_{J_{R,0}^{-\alpha_1/2}(a,b)}dy
\\
&~~~~+\frac{C}{\epsilon_{\alpha_2}}h^{2k+2}+C\epsilon_{\alpha_2}\int_a^b
\parallel\Pi e_{p_y}^n(x,\cdot)\parallel^2_{J_{R,0}^{-\alpha_2/2}(c,d)}dx,
\end{split}
\end{align*}
where $\epsilon_{\alpha_1}$ and $\epsilon_{\alpha_2}$ are chosen as sufficiently small numbers such that $C\epsilon_{\alpha_1}\leq \cos(\alpha_1\pi/2)$ and $C\epsilon_{\alpha_2}\leq \cos(\alpha_2\pi/2)$.

Integrating the first term of $\mathbbm{a}(\bm\Pi\bm e_{\bm\sigma}^n,\Pi u^n-u^n)$ by parts, and using
the orthogonal property of projection operator $\bm\Pi$, we get
\begin{align*}
{T}_3&=\big|\mathbbm{a}(\bm\Pi\bm e_{\bm\sigma}^n,\Pi u^n-u^n)\big|
\\
&=\big|\big(\bm\Pi\bm e_{\bm\sigma}^n,\nabla(\Pi u^n-u^n)\big)-\big(\{\bm\Pi\bm e_{\bm \sigma}^n\}\cdot \bm n_e,\llbracket\Pi u^n-u^n\rrbracket\big)_{\mathscr{E}_h^B}\big|
\\
&=\big|\big(\llbracket\bm\Pi\bm e_{\bm\sigma}^n\rrbracket,\{\Pi u^n-u^n\}\bm n_e\big)_{\mathscr{E}_h^i}\big|
\\
&\leq \sum_{e\in\mathscr{E}_h^i}\parallel\llbracket\bm\Pi\bm e_{\bm\sigma}^n\rrbracket\parallel_{L^2(e)}
\parallel\{\Pi u^n-u^n\}\bm n_e\parallel_{L^2(e)}
\\
&\leq \sum_{e\in\mathscr{E}_h^i}\left(\frac{1}{\epsilon_2}\parallel\{\Pi u^n-u^n\}\bm n_e\parallel_{L^2(e)}^2+
\frac{\epsilon_2}{4}\parallel\llbracket\bm\Pi\bm e_{\bm\sigma}^n\rrbracket\parallel_{L^2(e)}^2\right)
\\
&\leq \frac{C}{\epsilon_{2}}h^{2k+1}+\frac{\epsilon_{2}}{4}\sum_{e\in\mathscr{E}_h^i}\parallel
\llbracket\bm\Pi\bm e_{\bm\sigma}^n\rrbracket\parallel_{L^2(e)}^2.
\end{align*}
With the same deduction of ${T}_1$, there is
\begin{align*}
{T}_4&=\big|\mathbbm{d}(\Pi u^n-u^n,\Pi e_u^n)\big|
 \\
&\leq \epsilon_{1}\sum_{e\in\mathscr{E}_h^B}\parallel\llbracket\Pi u^n-u^n\rrbracket\parallel_{L^2(e)}\parallel\llbracket\Pi e_u^n\rrbracket\parallel_{L^2(e)}
\\
&\leq \epsilon_{1}\sum_{e\in\mathscr{E}_h^B}\left(\parallel\llbracket\Pi u^n-u^n\rrbracket\parallel_{L^2(e)}^2+\frac{1}{4}\parallel\llbracket\Pi e_u^n\rrbracket\parallel_{L^2(e)}^2\right)
\\
&\leq \epsilon_{1}Ch^{2k+1}+\frac{\epsilon_{1}}{4}\sum_{e\in\mathscr{E}_h^B}\parallel\llbracket\Pi e_u^n\rrbracket\parallel_{L^2(e)}^2.
\end{align*}
By Lemma \ref{lemma4.3}, we get
\begin{align*}
{T}_5&=\big|\mathbbm{e}(\bm\Pi\bm \sigma^n-\bm\sigma^n,\bm\Pi\bm e_{\bm\sigma}^n)\big|
\\
&\leq \epsilon_{2}\sum_{e\in\mathscr{E}_h^i}\parallel\llbracket\bm\Pi\bm\sigma^n-\bm\sigma^n\rrbracket
\parallel_{L^2(e)}\parallel\llbracket\bm\Pi\bm e_{\bm\sigma}^n\rrbracket\parallel_{L^2(e)}
\\
&\leq \epsilon_{2}\sum_{e\in\mathscr{E}_h^i}\left(\parallel\llbracket\bm\Pi\bm\sigma^n-\bm\sigma^n\rrbracket
\parallel_{L^2(e)}^2+\frac{1}{4}\parallel\llbracket\bm\Pi\bm e_{\bm\sigma}^n\rrbracket\parallel_{L^2(e)}^2\right)
\\
&\leq C\epsilon_{2}h^{2k+1}+\frac{\epsilon_2}{4}\sum_{e\in\mathscr{E}_h^i}\parallel
\llbracket\bm\Pi\bm e_{\bm\sigma}^n\rrbracket\parallel_{L^2(e)}^2.
\end{align*}

Note that ${T}_6$ and ${T}_7$ vanish because of the orthogonal property of the projection $\bm\Pi$.
Substituting ${T}_i,i=1,\cdots,7$ into (\ref{eq4.18}), the desired result is obtained.
\end{proof}

\subsubsection{Error bounds}
Assuming that
the solution of (\ref{eq1}) is sufficiently regular, we have the following error estimates.

\begin{thm}\label{theorem4.8}
Let $(u^n,\bm\sigma^n,\bm p^n)$ be the exact solution of (\ref{eqloworder}), $(u_h^n,~\bm\sigma_h^n,~\bm p_h^n)$ the numerical solution of the fully discrete HDG scheme (\ref{eq3.5}). If $\bm b\in L^\infty(J;W^{1,\infty}(\Omega)^2)$, for any integer $N=1,2,\cdots$, there is
\begin{align}
\begin{split}
&\|u^N-u_h^N\|_{L^2(\Omega)}^2+\Delta t\sum_{n=1}^N\Big(\epsilon_1\sum_{e\in\mathscr{E}_h^B}
\|\llbracket u^n-u_h^n\rrbracket\|_{L^2(e)}^2+\epsilon_2\sum_{e\in\mathscr{E}_h^i}
\|\llbracket\bm\sigma^n-\bm\sigma_h^n\rrbracket\|_{L^2(e)}^2
\\
&2K_{\alpha_1}\int_c^d \|(p_x^n-p_{xh}^n)(\cdot,y)\|^{2}_{J_{R,0}^{-\alpha_1/2}(a,b)}dy+2K_{\alpha_2}\int_a^b \| (p_y^n-p_{yh}^n)(x,\cdot)\|^{2}_{J_{R,0}^{-\alpha_2/2}(c,d)}dx\Big)
\\
&\leq C(\Delta t)^2\sum_{n=1}^N\|\partial_{\tau\tau}u\|^2_{L^2({J}^n;L^2(\Omega))}+C\sum_{n=1}^N
\|\partial_t(\Pi u-u)\|_{L^2({J}^n;L^2(\Omega))}^2
\\
& ~~~~ +C_\epsilon h^{2k+1}
+C\Delta t\sum_{n=1}^N
\mid\Pi u^{n-1}-u^{n-1}\mid_{H^1(\Omega)}^2,
\end{split}
\end{align}
where $\alpha_1=2-\alpha$, $\alpha_2=2-\beta$,  $K_{\alpha_1}=\cos(\alpha_1\pi/2)-C\epsilon_{\alpha_1}\geq0, K_{\alpha_2}=\cos(\alpha_2\pi/2)-C\epsilon_{\alpha_2}\geq0$, $\epsilon_{\alpha_1}$ and $\epsilon_{\alpha_2}$ are chosen as above, $C_\epsilon$ is dependent of $\epsilon_1,~\epsilon_2$.
\end{thm}
\begin{proof}
Substituting the results of Theorem \ref{theorem4.6} and Theorem \ref{theorem4.7} into (\ref{erro}),
there is
\begin{align*}
\begin{split}
&\frac{1}{2\Delta t}\big(\parallel\Pi e_u^n\parallel^2_{L^2(\Omega)}-\parallel\Pi e_u^{n-1}\parallel^2_{L^2(\Omega)}\big)+\frac{\epsilon_1}{2}\sum_{e\in\mathscr{E}_h^B}\parallel
\llbracket\Pi e_u^n\rrbracket\parallel_{L^2(e)}^2
\\
&~~~~ +\frac{\epsilon_2}{2}\sum_{e\in\mathscr{E}_h^i}
\parallel\llbracket\bm\Pi\bm e_{\bm\sigma}^n\rrbracket\parallel_{L^2(e)}^2
+(\cos(\alpha_1\pi/2)-C\epsilon_{\alpha_1})\int_c^d
\|\Pi e_{p_x}^n(\cdot,y)\|_{J_{R,0}^{-\alpha_1/2}(a,b)}^2dy
\\
& ~~~~ +(\cos(\alpha_2\pi/2)-C\epsilon_{\alpha_2})\int_a^b
\|\Pi e_{p_y}^n(x,\cdot)\|_{J_{R,0}^{-\alpha_2/2}(c,d)}^2dx
\\
&\leq C\parallel\Pi e_u^{n-1}\parallel^2_{L^2(\Omega)}+C\parallel\Pi e_u^n\parallel^2_{L^2(\Omega)}+C\Delta t\parallel\partial_{\tau\tau}u\parallel^2_{L^2(J^n;L^2(\Omega))}
\\
&~~~~ +\frac{C}{\Delta t}\|\partial_t(\Pi u-u)\|^2_{L^2(J^n;L^2(\Omega))}
+C\mid\Pi u^{n-1}-u^{n-1}\mid^2_{H^1(\Omega)}+C_\epsilon h^{2k+1}.
\end{split}
\end{align*}
With $\Pi e_u^0=0$, multiplying the above inequality by $2\Delta t$ on both sides, summing
over $n$ from $1$ to $N$, and using the discrete Gr\"{o}nwall inequality, there is
\begin{align*}
&\parallel\Pi e_u^N\parallel^2_{L^2(\Omega)}+\Delta t\sum_{n=1}^N\big(\epsilon_1\sum_{e\in\mathscr{E}_h^B}\parallel
\llbracket\Pi e_u^n\rrbracket\parallel_{L^2(e)}^2+\epsilon_2\sum_{e\in\mathscr{E}_h^i}
\parallel\llbracket\bm\Pi\bm e_{\bm\sigma}^n\rrbracket\parallel_{L^2(e)}^2\big)
\\
&~~~~ +2\Delta t\sum_{n=1}^N(cos(\alpha_1\pi/2)-C\epsilon_{\alpha_1})\int_c^d
\parallel\Pi e_{p_x}^n(\cdot,y)\parallel_{J_{R,0}^{-\alpha_1/2}(a,b)}^2dy
\\
&~~~~ +2\Delta t\sum_{n=1}^N(cos(\alpha_2\pi/2)-C\epsilon_{\alpha_2})\int_a^b
\|\Pi e_{p_y}^n(x,\cdot)\|_{J_{R,0}^{-\alpha_2/2}(c,d)}^2dx
\\
&\leq C(\Delta t)^2\sum_{n=1}^N\|\partial_{\tau\tau}u\|^2_{L^2(J^n;L^2(\Omega))}
+C\sum_{n=1}^N\|\partial_t(\Pi u-u)\|^2_{L^2(J^n;L^2(\Omega))}
\\
&~~~~ +C\Delta t\sum_{n=1}^N\mid\Pi u^{n-1}-u^{n-1}\mid^2_{H^1(\Omega)}+C_\epsilon h^{2k+1}.
\end{align*}
By the triangle inequality, we obtain the desired result.
\end{proof}

\section{Numerical experiment}
In this section, we illustrate the numerical performance of the proposed schemes by the numerical simulations of two
examples. In the first example, we take the vector function $\bm b=\bm 0$ and verify the accuracy of the schemes with the exact smooth solution $u$ combining with the left fractional Riemann-Liouville  derivatives with respect to $x$-variable and $y$-variable, respectively. When we compute the fractional integral part in triangular meshes (see Figures 1-2), the Gauss points and weights are used to deal with the terms relating with the fractional operators element-by-element (see \cite{r8,r30}). Since this part needs more time and memory spaces (see \cite{r22}), we only use the piecewise linear basis functions to simulate the solution in triangular meshes. Tables 1-3 illustrate that the schemes have a good convergence order with
piecewise linear basis function for different choices of the fluxes. In the second example, we take $\bm b$ to be a vector function and perform some numerical experiments with some figures (see Figures 3-4)
which justify that the schemes simulate the solution very well for 2D-fractional convection-diffusion problems.

 \begin{figure}[H]
\begin{minipage}[t]{0.48\linewidth}
\centering
\includegraphics[width=1.8in,height=1.8in]{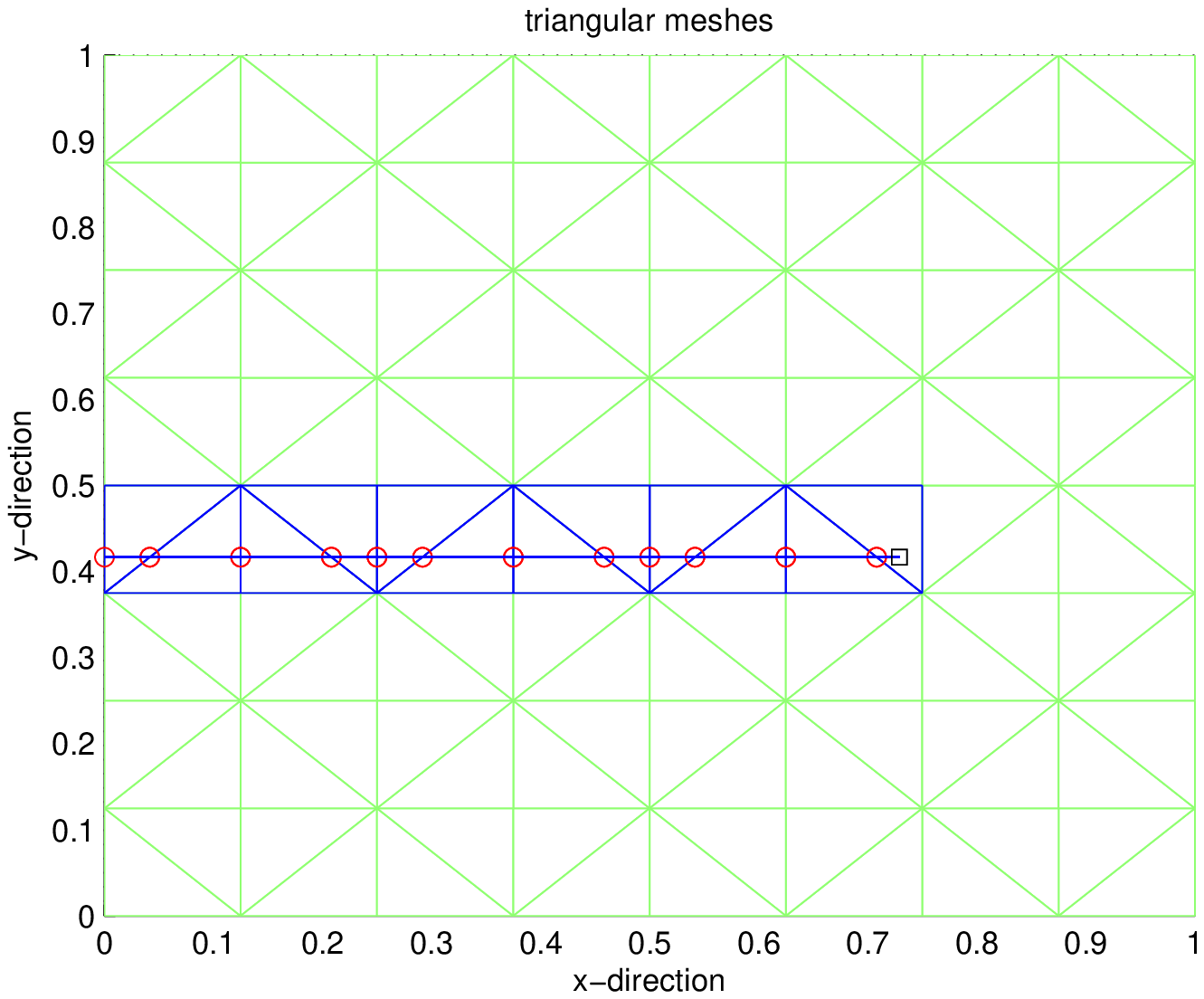}
\caption{All triangles in x-direction affected by the Gauss points (denoted by black square).}
\end{minipage}
\begin{minipage}[t]{0.48\linewidth}
\centering
\includegraphics[width=1.8in,height=1.8in]{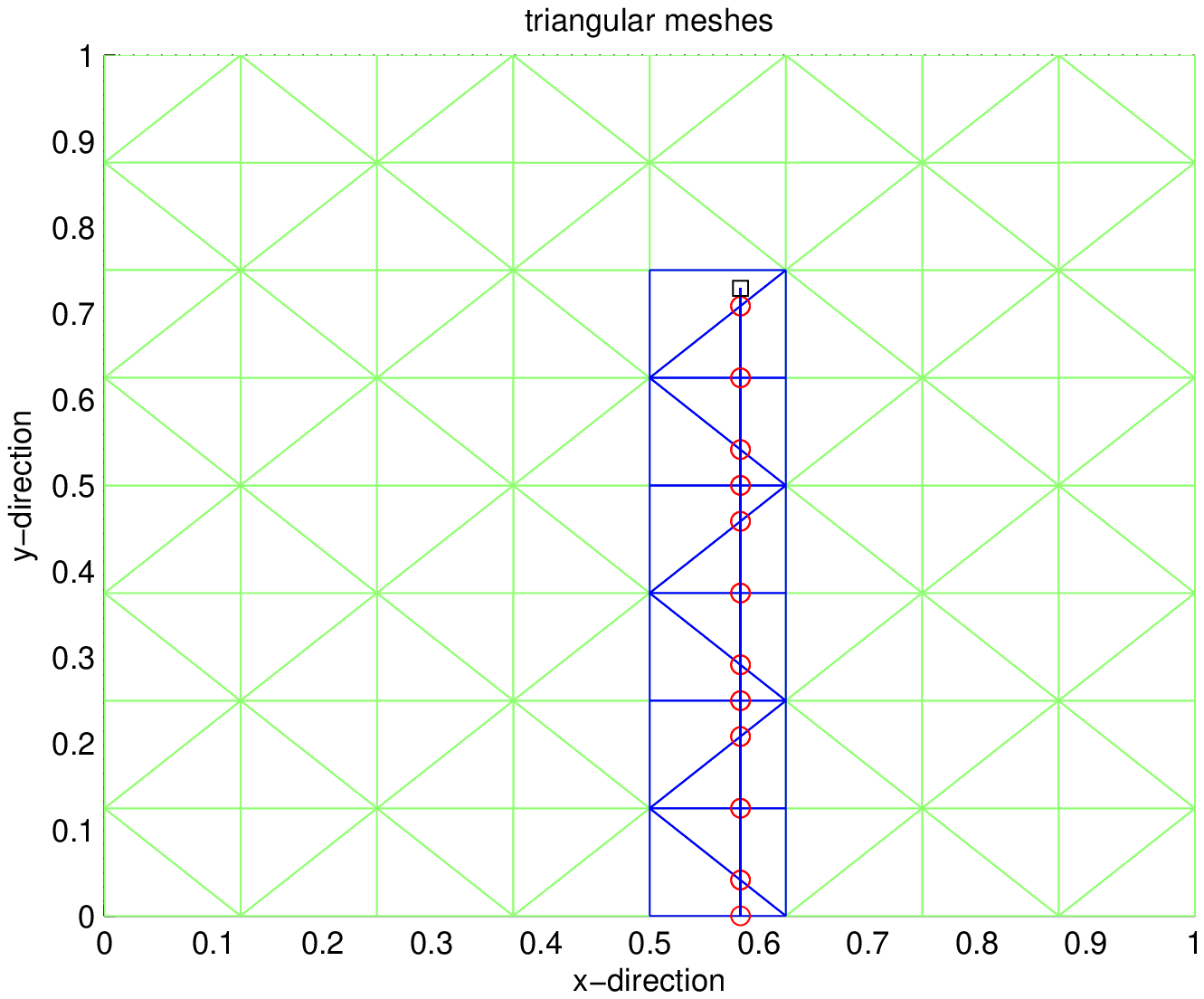}
\caption{All triangles in y-direction affected by the Gauss points (denoted by black square).}
\end{minipage}
\end{figure}
$\textbf{Example 5.1.}$ Consider 2D space-fractional convection-diffusion problem (\ref{eq1}) in domain $\Omega =(0,1)\times(0,1)$. The initial condition and the exact solution are specified as
\begin{equation}
\left\{ \begin{array}
 {l@{\quad} l}
 u(x,y,t)=e^{-t}x^2(x-1)^2y^2(y-1)^2,
 \\
 u_{0}(x,y)=x^2(x-1)^2y^2(y-1)^2,
 \\
 \bm b(x,y,t)=(0,0).
 \end{array}
 \right.
\end{equation}
Then the force term $f$ is determined accordingly from (1.1). In this case, we present a few results to numerically validate the analysis.

For the numerical simulations, in order to validate the stability and the accuracy of the presented HDG scheme, we choose the time-stepsize, $\Delta t=\mathcal{O}(h^{3/2})$, used to advance the discrete formulation from $t^{n-1}$ to $t^n, n=1,2,...,N$. The experimental convergence rate is given by
 $$rate=\frac{\log\big(\parallel u(t)-u_{h_1}(t)\parallel_{L^2(\mathscr{E}_{h_1})}/\parallel u(t)-u_{h_2}(t)\parallel_{L^2(\mathscr{E}_{h_2})}\big)}{\log(h_1/h_2)}.$$

\begin{table}[H]
\centering
  \begin{tabular}{|c|c|c|c|c|c|c|c|}  \hline
  \multicolumn{1}{|c|}{}
&\multicolumn{7}{|c|}{$t=0.1,(\alpha,\beta)=(1.2,1.4),(\epsilon_1,\epsilon_2)=(\mathcal{O}(1),\mathcal{O}(1))$}
\\
\hline
$h$&$\|e_u(t)\|_{L^2}$& $rate$   &$\|e_u(t)\|_{L^1}$& $rate$
   &$\|\partial_x e_u(t)\|_{L^2}$& $rate$   &$\|\partial_y e_u(t)\|_{L^2}$\\ \hline
$1/6$&1.3993e-04&--   &1.0771e-04&--  &1.4273e-03& --  &6.4132e-03 \\

$1/10$&5.6088e-05&1.79 &4.1112e-05&1.89&7.2047e-03&1.34 &6.8133e-03\\

$1/14$&2.9803e-05&1.88 &2.1713e-05&1.90&4.8274e-04&1.19 &6.9191e-03\\

$1/18$&1.8452e-05&1.91 &1.3226e-05&1.97&3.4904e-04&1.29 &6.9660e-03\\
\hline
&\multicolumn{7}{|c|}{$t=0.1,(\alpha,\beta)=(1.5,1.5),(\epsilon_1,\epsilon_2)=(\mathcal{O}(1),\mathcal{O}(1))$}
\\
 \hline
$1/6$&1.8283e-04&--    &1.4041e-04&--  &1.1688e-03& --  &6.4118e-03 \\

$1/10$&7.5004e-05&1.74 &5.6669e-05&1.78&5.1740e-04&1.60 &6.7972e-03\\

$1/14$&4.0967e-05&1.80 &3.1156e-05&1.78&3.2141e-04&1.41 &6.9127e-03\\

$1/18$&2.5475e-05&1.89 &1.9258e-05&1.91&2.2281e-04&1.46 &6.9610e-03\\

\hline
  \multicolumn{1}{|c|}{}
&\multicolumn{7}{|c|}{$t=0.1,(\alpha,\beta)=(1.9,1.6),(\epsilon_1,\epsilon_2)=(\mathcal{O}(1),\mathcal{O}(1))$}
\\
 \hline
$1/6$ &2.6485e-04&--    &1.9290e-04&--  &1.1859e-03& --  &6.4542e-03 \\

$1/10$&1.1544e-04&1.63  &8.6542e-05&1.57&5.3599e-04&1.56 &6.7608e-03\\

$1/14$&6.7712e-05&1.59  &5.0400e-05&1.61&3.3163e-04&1.43 &6.8829e-03\\

$1/18$&4.5463e-05&1.59  &3.3857e-05&1.58&2.3476e-04&1.38 &6.9387e-03\\
\hline
 \end{tabular}
 \caption{The $L^2,L^1$-errors and convergence rates for $u$ and $u_x,u_y$ for Example 5.1.}
 \end{table}
\begin{table}[H]
\centering
  \begin{tabular}{|c|c|c|c|c|c|c|c|}  \hline
  \multicolumn{1}{|c|}{}
&\multicolumn{7}{|c|}{$t=1,(\alpha,\beta)=(1.2,1.4),(\epsilon_1,\epsilon_2)=(\mathcal{O}(1),\mathcal{O}(1))$}
\\
 \hline
$h$&$\| e_u(t)\|_{L^2}$& $rate$   &$\|e_u(t)\|_{L^1}$& $rate$
   &$\| \partial_x e_u(t)\|_{L^2}$& $rate$   &$\| \partial_y e_u(t)\|_{L^2}$\\ \hline
$1/6$  &8.2881e-05&--   &7.2374e-05&--  &4.7733e-04& --  &2.5556e-03\\

$1/10 $&3.2222e-05&1.85 &2.7700e-05&1.88&2.6204e-04&1.17 &2.7560e-03\\

$1/14$ &1.6162e-05&2.05 &1.3515e-05&2.13&1.7748e-04&1.06 &2.8077e-03\\

$1/18$ &9.9448e-06&1.93 &8.2787e-06&1.95&1.3085e-04&1.21 &2.8291e-03\\
\hline
  \multicolumn{1}{|c|}{}
&\multicolumn{7}{|c|}{$t=1,(\alpha,\beta)=(1.5,1.5),(\epsilon_1,\epsilon_2)=(\mathcal{O}(1),\mathcal{O}(1))$}
\\
 \hline
$1/6$ &8.7928e-05&--   &7.4712e-05&--  &4.1510e-04& --  &2.5466e-03 \\

$1/10$&3.5668e-05&1.77 &2.9706e-05&1.81&1.9555e-04&1.47 &2.7408e-03\\

$1/14$&1.8524e-05&1.95 &1.4885e-05&2.05&1.2291e-04&1.38 &2.8010e-03\\

$1/18$&1.1432e-05&1.92 &9.0662e-06&1.97&8.6553e-05&1.40 &2.8244e-03\\
\hline
  \multicolumn{1}{|c|}{}
&\multicolumn{7}{|c|}{$t=1,(\alpha,\beta)=(1.9,1.6),(\epsilon_1,\epsilon_2)=(\mathcal{O}(1),\mathcal{O}(1))$}
\\
 \hline
$1/6$ &1.1250e-04&--   &8.4393e-05&--  &4.6409e-04& --  &2.5983e-03 \\

$1/10$&4.7998e-05&1.67 &3.6510e-05&1.64&2.1614e-04&1.50 &2.7410e-03\\

$1/14$&2.7916e-05&1.61 &2.0971e-05&1.65&1.3512e-04&1.40 &2.7985e-03\\

$1/18$&1.8767e-05&1.58 &1.4104e-05&1.58&9.5762e-05&1.37 &2.8216e-03\\
\hline
\end{tabular}
\caption{The $L^2,L^1$-errors and convergence rates for $u$ and $u_x,u_y$ for Example 5.1.}
\end{table}
\begin{table}[H]
\centering
  \begin{tabular}{|c|c|c|c|c|c|c|c|}  \hline
  \multicolumn{1}{|c|}{}
&\multicolumn{7}{|c|}{$t=1,(\alpha,\beta)=(1.9,1.6),(\epsilon_1,\epsilon_2)=(\mathcal{O}(h^{-1}),\mathcal{O}(1))$}
\\
 \hline
$h$&$\| e_u(t)\|_{L^2}$& $rate$   &$\| e_u(t)\|_{L^1}$& $rate$
   &$\|\partial_x e_u(t)\|_{L^2}$& $rate$   &$\|\partial_y e_u(t)\|_{L^2}$\\ \hline
$1/6$  &5.2142e-05&--   &4.1592e-05&--  &4.0695e-04& --  &2.5777e-03\\

$1/10 $&1.9772e-05&1.90 &1.5784e-05&1.90&1.7761e-04&1.62 &2.7383e-03\\

$1/14$ &9.5805e-06&2.15 &7.5267e-06&2.20&1.1392e-04&1.32 &2.7976e-03\\

$1/18$ &5.8213e-06&1.98 &4.6596e-06&1.91&7.8388e-05&1.49 &2.8210e-03\\

\hline
  \multicolumn{1}{|c|}{}
&\multicolumn{7}{|c|}{$t=1,(\alpha,\beta)=(1.9,1.6),(\epsilon_1,\epsilon_2)=(\mathcal{O}(h^{-1}),\mathcal{O}(h))$}
\\
 \hline
$1/6$ &4.8702e-05&--   &3.9358e-05&--  &4.2857e-04& --  &2.5666e-03 \\

$1/10$&1.9169e-05&1.83 &1.5766e-05&1.79&1.9937e-04&1.50 &2.7356e-03\\

$1/14$&9.2520e-06&2.17 &7.6271e-06&2.16&1.2817e-04&1.31 &2.7962e-03\\

$1/18$&5.6525e-06&1.96 &4.6650e-06&1.96&8.9725e-05&1.42 &2.8202e-03\\
\hline
  \multicolumn{1}{|c|}{}
&\multicolumn{7}{|c|}{$t=1,(\alpha,\beta)=(1.9,1.6),(\epsilon_1,\epsilon_2)=(\mathcal{O}(1),\mathcal{O}(h))$}
\\
 \hline
$1/6$ &8.6286e-05&--   &6.6645e-05&--  &4.5649e-04& --  &2.5761e-03 \\

$1/10$&3.4635e-05&1.79 &2.8021e-05&1.70&2.0660e-04&1.55 &2.7365e-03\\

$1/14$&1.7787e-05&1.98 &1.4393e-05&1.98&1.2993e-04&1.38 &2.7964e-03\\

$1/18$&1.0969e-05&1.92 &8.9170e-06&1.91&8.9646e-05&1.48 &2.8200e-03\\
\hline
\end{tabular}
\caption{The $L^2,L^1$-errors and convergence rates for $u$ and $u_x,u_y$ for Example 5.1.}
\end{table}
In Table 1 and Table 2 we choose different observation time $t=0.1,1$ and $\alpha,\beta$ to justify that the
convergence rates at least have an order of $\mathcal{O}(h^{3/2})$ for the solution $u$ in $L^2,~L^1$-norms
based on the piecewise linear basis function. In Table 3 we take the same choice of $\epsilon_1,\epsilon_2$ as Ref. \cite{r4} and see that the convergence rates increase to $\mathcal{O}(h^2)$ (see the explanations in Ref. \cite{r4}). Comparing the numerical results with the work \cite{r30}, we can see that the HDG method has smaller numerical errors for the first order polynomial approximation.

\textbf{Example 5.2.} In this example, we investigate the approximation solution of problem (\ref{eq1}). For
convenience, we still choose the domain $\Omega =(0,1)\times(0,1)$. The exact solution $u$, initial value and the vector function $\bm b$ are given by
 \begin{equation}
 \begin{cases}
 {u(x,y,t)=}
 e^{-t}x^2(x-0.5)^2(x-1)^2y^2(y-0.5)^2(y-1)^2,
 \\
 u_0(x,y)=x^2(x-0.5)^2(x-1)^2y^2(y-0.5)^2(y-1)^2,
 \\
\bm b=((x-0.5),-(y-0.5)).
 \end{cases}
 \end{equation}
For the second example, in order to further support the theoretical convergence and justify
the powerful HDG scheme, we take $\bm b$ to be nonzero vector function and give some approximation solutions with the refining space-step $h$ to compare with the exact solutions and display the efficiency of the simulations.


\begin{figure}[H]
\subfigure[]{
\begin{minipage}[t]{0.3\linewidth}
\centering
\includegraphics[width=1.1in,height=1.1in]{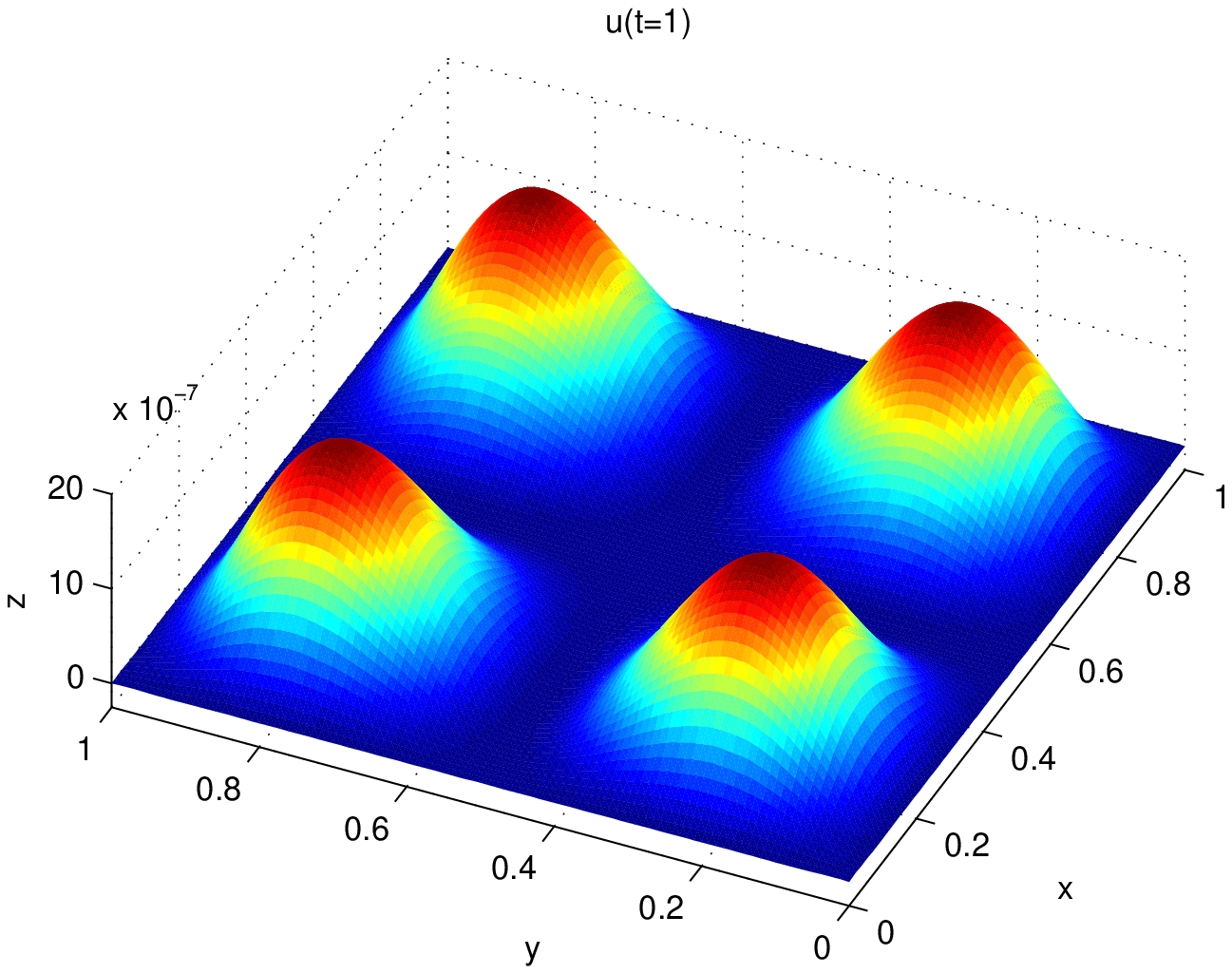}
\end{minipage}
}
\subfigure[]{
\begin{minipage}[t]{0.3\linewidth}
\centering
\includegraphics[width=1.10in,height=1.10in]{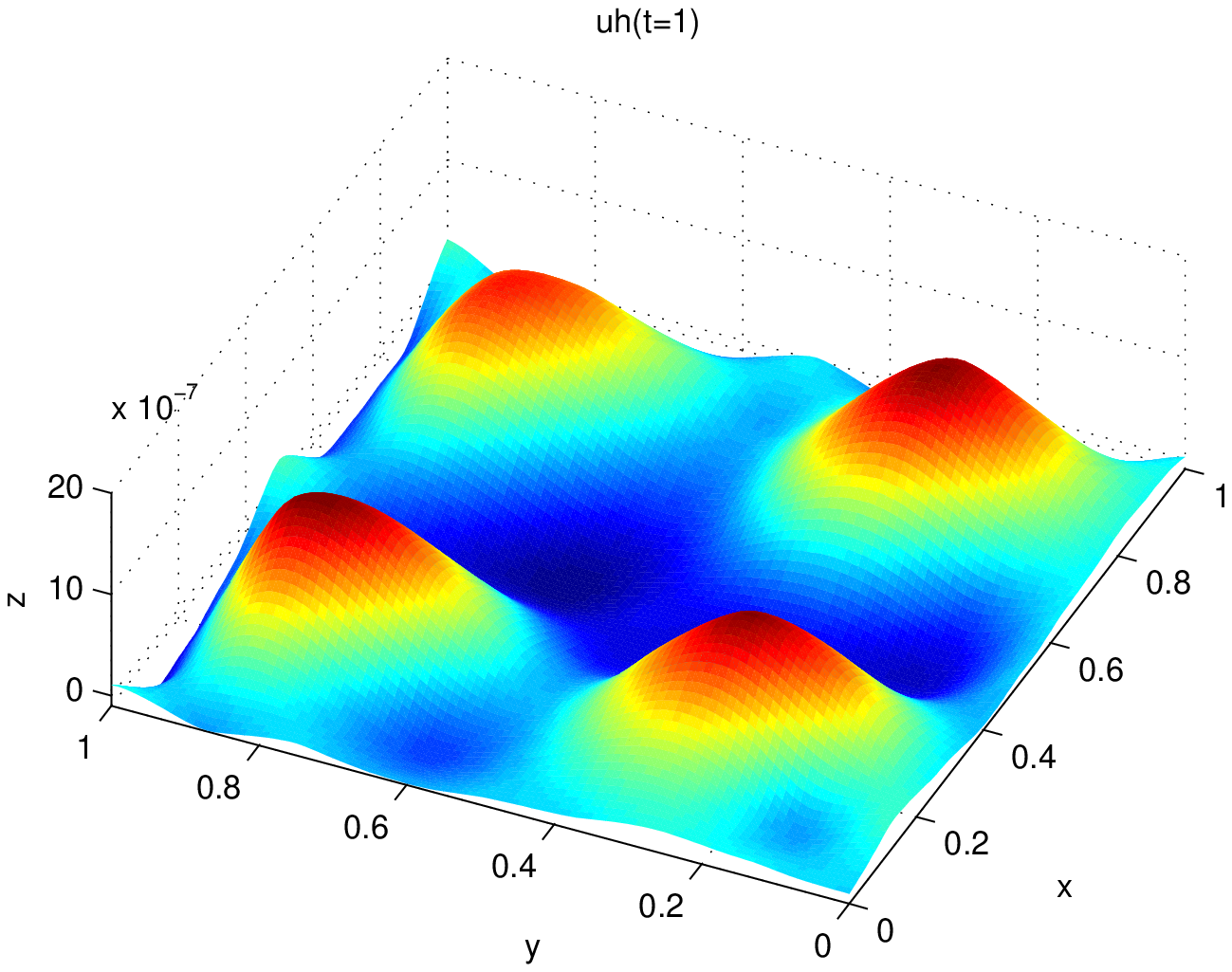}
\end{minipage}
}
\subfigure[]{
\begin{minipage}[t]{0.3\linewidth}
\centering
\includegraphics[width=1.10in,height=1.10in]{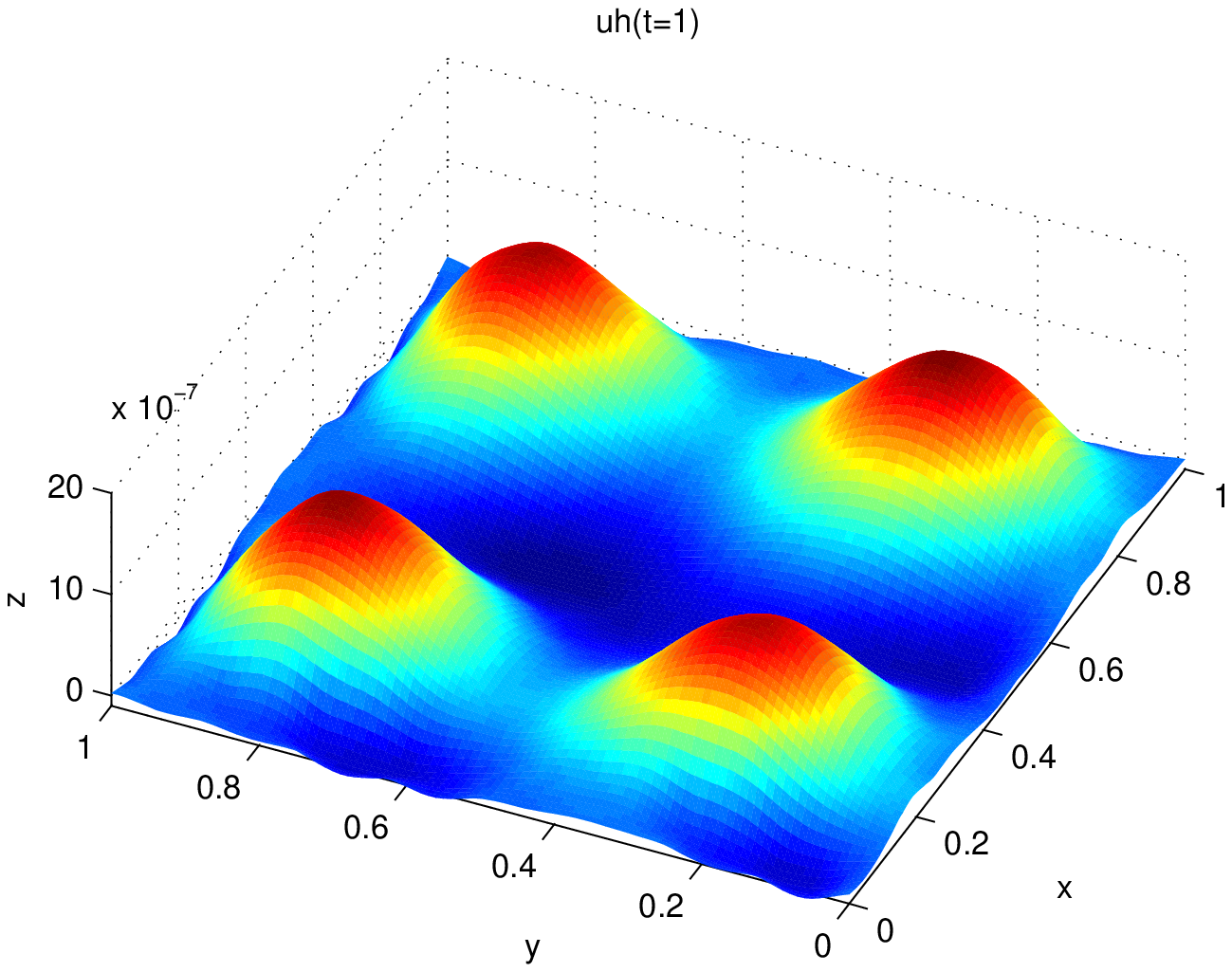}
\end{minipage}
}
\caption{Exact solution $u$ and the numerical solutions $u_h$ at $t=1$ for Example 5.2.} \label{fig:1}
\end{figure}

\begin{figure}[H]
\subfigure[]{
\begin{minipage}[t]{0.3\linewidth}
\centering
\includegraphics[width=1.1in,height=1.1in]{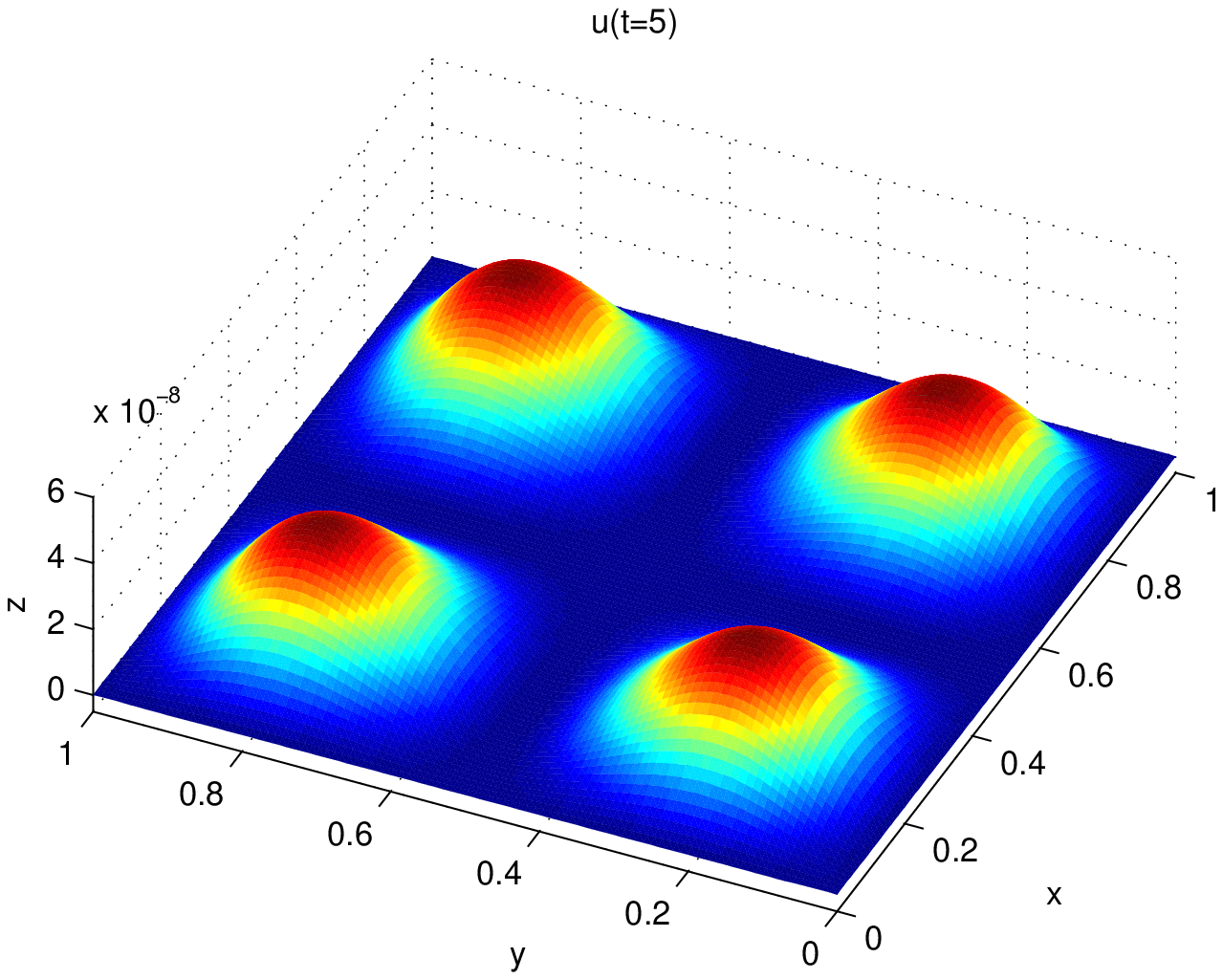}
\end{minipage}
}
\subfigure[]{
\begin{minipage}[t]{0.3\linewidth}
\centering
\includegraphics[width=1.10in,height=1.10in]{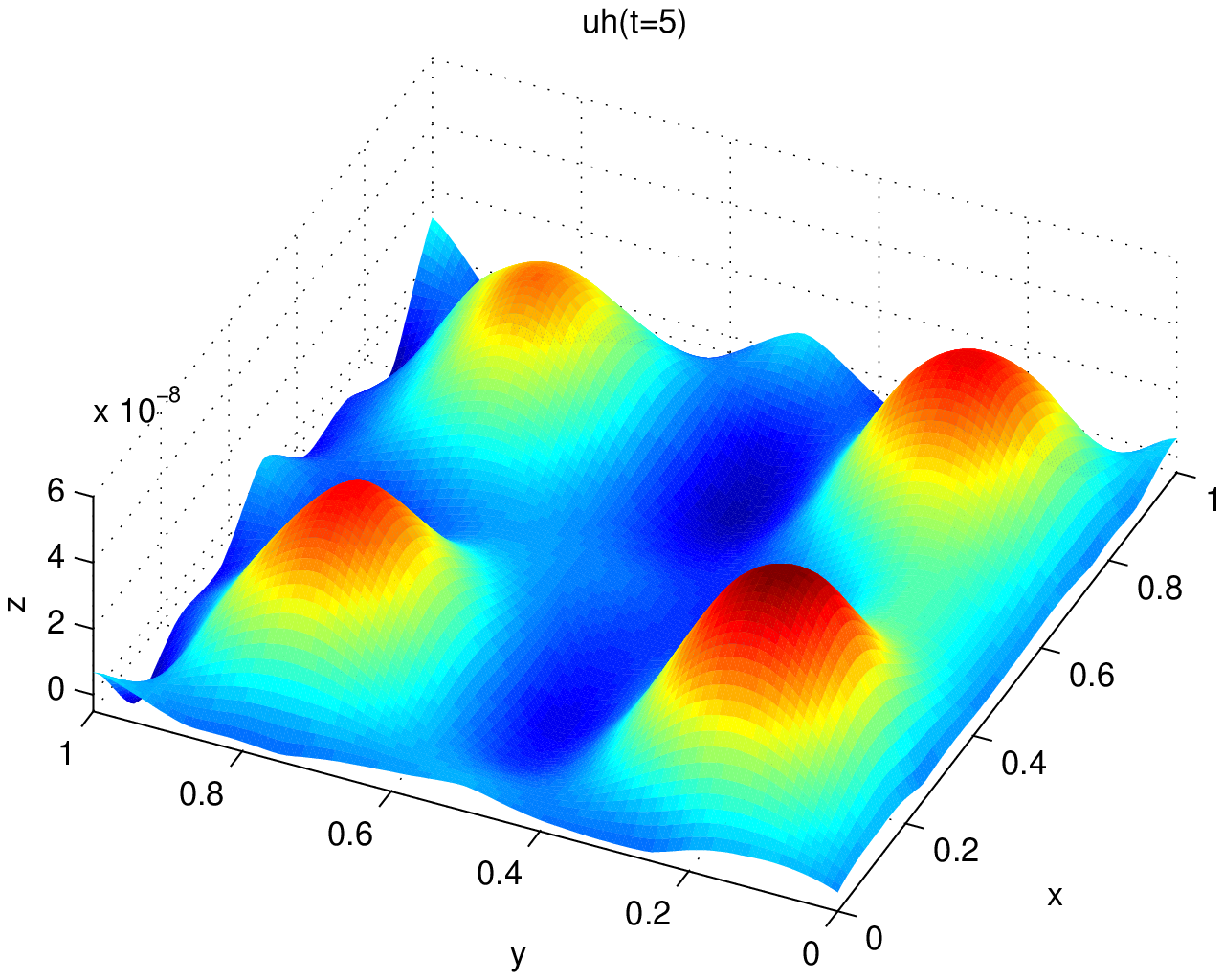}
\end{minipage}
}
\subfigure[]{
\begin{minipage}[t]{0.3\linewidth}
\centering
\includegraphics[width=1.10in,height=1.10in]{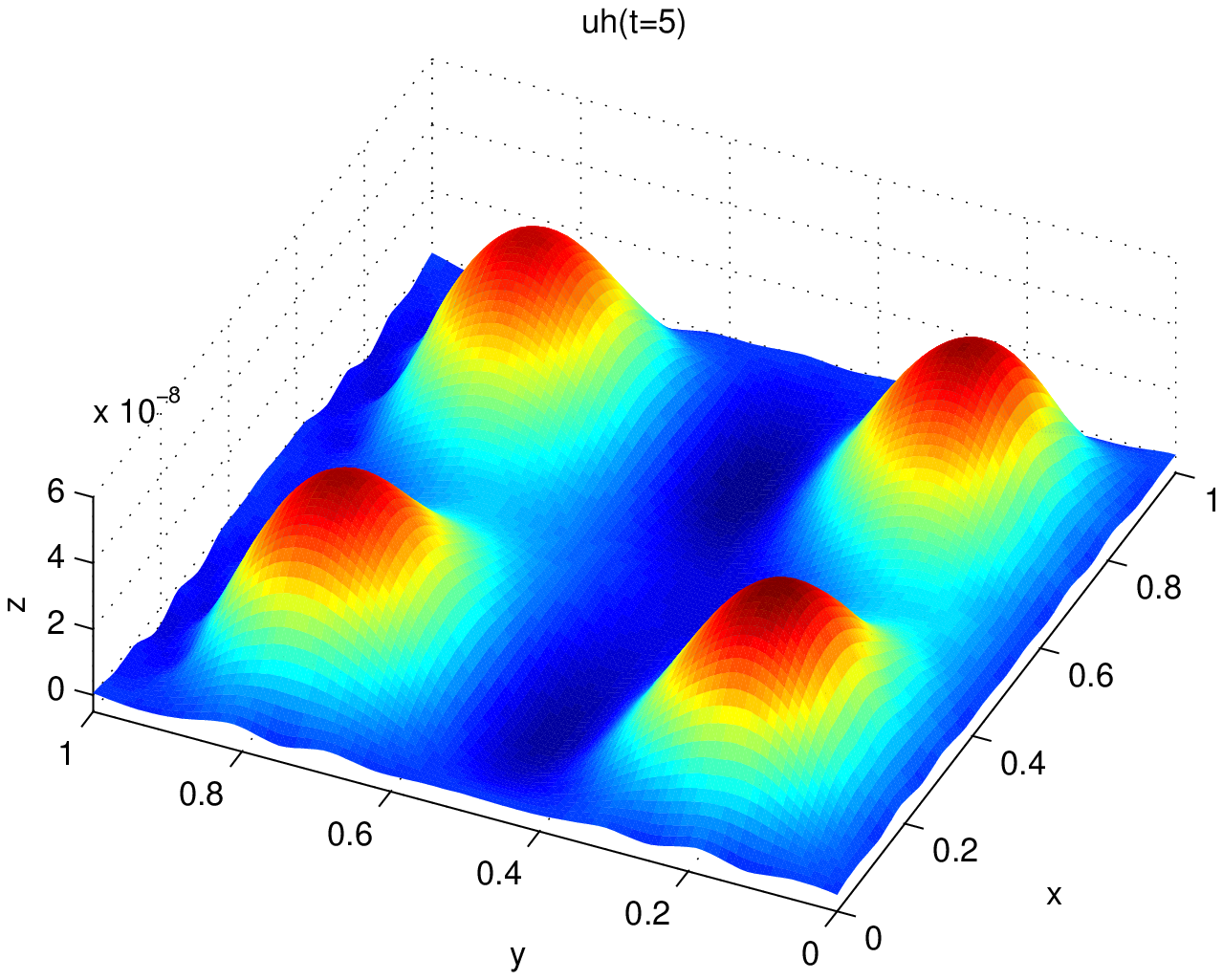}
\end{minipage}
}
\caption{Exact solution $u$ and the numerical solutions $u_h$ at $t=5$ for Example 5.2.} \label{fig:1}
\end{figure}
Figure 3 displays the exact solution $u$ and the numerical solutions $u_h$ based on different space stepsizes $h=\frac{1}{8},\frac{1}{16}$ at $t=1$ with $\alpha=1.2,\beta=1.4$, $\epsilon_1=\epsilon_2=1$. Figure 4 displays the exact solution $u$ and the numerical solutions $u_h$ based on different space stepsizes $h=\frac{1}{8},\frac{1}{16}$
at $t=5$ with $\alpha=1.9,\beta=1.6,\epsilon_1=h^{-1},\epsilon_2=1$. It is clear that the exact solution of Example 5.2 is nonnegative with four hills. In the simulations, the $P^1$-HDG solutions recover the exact solution perfectly with all four hills in coarse meshes. Note that the numerical results display that the approximations are more and more accurate with the refining of the meshes.
%
\section{Conclusions}
By carefully introducing the auxiliary variables, constructing the numerical fluxes, adding the penalty terms, and using the characteristic method to deal with the time derivative and convective term, we design the effective HDG schemes to solve 2D space-fractional convection-diffusion equations with triangular meshes. As we know, this work is the first time to deal two-dimensional space-fractional convection-diffusion equation with triangular mesh by the DG method. The stability and error bounds analysis are investigated.

Besides the general advantages of HDG method, the presented scheme is shown to have the following benefits: 1) it is symmetric, so easy to deal with the fractional operators; 2) theoretically, the stability can be more easily proved; 3) the penalty terms make the error analysis more convenient; 4) numerically verified to have efficient approximations; 5) the schemes are performed very well in triangular meshes; 6) it is possible to use this scheme to solve nonlinear equations which is the future research task.

\begin{acknowledgements}
This work was partially supported by the National Basic Research (973) Program
 of China under Grant 2011CB706903, the National Natural Science Foundation of China under Grant 11271173 and 11471150, and the CAPES and CNPq in Brazil.
\end{acknowledgements}

\end{document}